%% file: main.tex
\documentclass{amsart}
\input{preamble}

\makeatletter
\let\@wraptoccontribs\wraptoccontribs
\makeatother

\title{A categorical approach to injective envelopes}
\author{Arianna Cecco}

\address{Department of Mathematics, University of Houston, USA}
\dedicatory{with an appendix by David P. Blecher}
\email{ahcecco@central.uh.edu}
\address{Department of Mathematics, University of Houston, USA}
\email{dpbleche@central.uh.edu}

\thanks{AC was supported by the NSF Grant DMS-2155162.}

\begin{document}

%-----Title-----%

\maketitle

%-----Abstract-----%

\begin{abstract}
    We explore functors between operator space categories, some properties of these functors, and establish relations between objects in these categories and their images under these functors, in particular regarding injectivity and injective envelopes. We also compare the purely categorical definition of injectivity with the `standard' operator theoretical definition. An appendix by D. P. Blecher discusses the unitization of an operator space and its injective envelope.
\end{abstract}

%-----Introduction-----%

\section{Introduction}

The study of injective envelopes in an `operator' or $C^*$-setting started with Hamana in the `70s, with \cite{hamana79a,hamana79b}, and has many applications. More recent applications (see \cite{kk14, bkko, raum20}) have led to discovering many important properties of groups and group $C^*$-algebras. Since injectivity is categorical in nature, it is reasonable to ask if a similar theory can be developed for other `operator space categories'; for example, a category consisting of real operator spaces (which are currently rising in interest). Further, there are natural functors that arise between these operator space categories, and it makes sense to ask if these functors will preserve such structure. The main concern of this paper is answering these question and in fact, the answer is in the affirmative.  The results can be summarized in an all-encompassing theorem:

\begin{theorem}\label{main}
    Let $\mcf$ be any functor from the diagram below, let $X$ be an object from the domain category of $\mcf$, and let $I(X)$ be an injective envelope of $X$ in that category. Then it follows that $$I(\mcf(X))=\mcf(I(X)).$$
\end{theorem}

The following diagram contains the categories and natural functors that we will be concerned with in this paper (this diagram will be referred to as the `functorial picture'):

\begin{center}
    \begin{tikzcd}
        %row 1
        %1c1
        \mathsf{OSp}^G_{(\bbr,\bbr)}
        
        \arrow[from=1-1,to=1-3, "\textcolor{red}{\mcf_c*}", shift left,labels=above]
        
        \arrow[from=1-1,to=2-1,"\textcolor{violet}{\mcf_u}",labels=right, shift left]

        %1c3
        
        & & \mathsf{OSp}^G_{(\bbc,\bbc)}
        
        \arrow[from=1-3,to=1-1,"\textcolor{red}{\mcf_0*}", shift left,labels=below]
        
        \arrow[from=1-3, to=1-5, "\mcf_0", shift left, labels=above]

        \arrow[from=1-3,to=2-3,"\textcolor{violet}{\mcf_u}",labels=right, shift left]
        
        %1c5
        & & \mathsf{OSp}^{G}_{(\bbc,\bbr)}\\

        \arrow[from=1-5,to=1-3,"\mcf_c", shift left, labels=below]

        \arrow[from=1-5,to=2-5,"\textcolor{violet}{\mcf_u}",labels=right, shift left]

        %row 2
        %2c1
        \mathsf{OSp}^{\1,G}_{(\bbr,\bbr)}
        
        \arrow[from=2-1,to=2-3, "\textcolor{red}{\mcf_c*}",shift left,labels=above]
        
        \arrow[from=2-1,to=1-1,"\mcf_0",labels=left, shift left]

        \arrow[from=2-1,to=3-1,"\textcolor{red}{\mcf_A*}",labels=right, shift left]
        
        %2c3
        & & \mathsf{OSp}^{\1,G}_{(\bbc,\bbc)}
        
        \arrow[from=2-3,to=2-1,"\textcolor{red}{\mcf_0*}", shift left,labels=below]
        
        \arrow[from=2-3, to=2-5, "\mcf_0", shift left, labels=above]

        \arrow[from=2-3,to=3-3,"\textcolor{red}{\mcf_A*}",labels=right, shift left]

        \arrow[from=2-3,to=1-3,"\mcf_0",labels=left, shift left]
        
        %2c5
        & & \mathsf{OSp}^{\1,G}_{(\bbc,\bbr)}
        
        \arrow[from=2-5,to=2-3,"\mcf_c", shift left, labels=below]

        \arrow[from=2-5,to=3-5,"\mcf_A",labels=right, shift left]

        \arrow[from=2-5,to=1-5,"\mcf_0",labels=left, shift left]
        
        \\

        %row 3
        %3c1
        \mathsf{OSy}^{G}_{(\bbr,\bbr)}
        
        \arrow[from=3-1,to=3-3, "\textcolor{red}{\mcf_c*}",shift left,labels=above]

        \arrow[from=3-1,to=2-1,"\mcf_0",labels=left, shift left]

        \arrow[from=3-1,to=4-1,"\textcolor{red}{\mcf_0*}",labels=right, shift left]
        
        %3c3
        & & \mathsf{OSy}^{G}_{(\bbc,\bbc)}
        
        \arrow[from=3-3,to=3-1,"\textcolor{red}{\mcf_0*}", shift left,labels=below]
        
        \arrow[from=3-3, to=3-5, "\mcf_0", shift left, labels=above]

        \arrow[from=3-3,to=4-3,"\textcolor{red}{\mcf_0*}",labels=right, shift left]

        \arrow[from=3-3,to=2-3,"\mcf_0",labels=left, shift left]
        
        %3c5
        & & \mathsf{OSy}^G_{(\bbc,\bbr)}
        
        \arrow[from=3-5,to=3-3,"\mcf_c", shift left, labels=below]

        \arrow[from=3-5,to=4-5,"\mcf_0",labels=right, shift left]

        \arrow[from=3-5,to=2-5,"\mcf_0",labels=left, shift left]
        \\
        
        %row 4
        %4c1
        \mathsf{OSy}_{(\bbr,\bbr)}
        
        \arrow[from=4-1,to=4-3, "\textcolor{red}{\mcf_c*}",shift left,labels=above]

        \arrow[from=4-1,to=3-1,"\mcf_G",labels=left, shift left]

        \arrow[from=4-1,to=5-1,"\mcf_0",labels=right, shift left]
        
        %4c3
        & & \mathsf{OSy}_{(\bbc,\bbc)}
         \arrow[from=4-3,to=4-1,"\textcolor{red}{\mcf_0*}", shift left,labels=below]
         
        \arrow[from=4-3, to=4-5, "\mcf_0", shift left, labels=above]

        \arrow[from=4-3,to=5-3,"\mcf_0",labels=right, shift left]

        \arrow[from=4-3,to=3-3,"\mcf_G",labels=left, shift left]
        
        %4c5
        & & \mathsf{OSy}_{(\bbc,\bbr)}
        
        \arrow[from=4-5,to=4-3,"\mcf_c", shift left, labels=below]

        \arrow[from=4-5,to=5-5,"\mcf_0",labels=right, shift left]

        \arrow[from=4-5,to=3-5,"\mcf_G",labels=left, shift left]
        \\
        
        %row 5
        %5c1
        \mathsf{OSp}^{\1}_{(\bbr,\bbr)}
        
         \arrow[from=5-1,to=5-3, "\textcolor{red}{\mcf_c*}",shift left,labels=above]

         \arrow[from=5-1,to=4-1,"\textcolor{red}{\mcf_A*}",labels=left, shift left]

         \arrow[from=5-1,to=6-1,"\mcf_0",labels=right, shift left]
        
        %5c3
        & & 
        \mathsf{OSp}^{\1}_{(\bbc,\bbc)}
        
        \arrow[from=5-3,to=5-1,"\textcolor{red}{\mcf_0*}", shift left,labels=below]
        
        \arrow[from=5-3, to=5-5, "\mcf_0", shift left, labels=above]

        \arrow[from=5-3,to=6-3,"\mcf_0",labels=right, shift left]

        \arrow[from=5-3,to=4-3,"\textcolor{red}{\mcf_A*}",labels=left, shift left]

        %5c5
        & & 
        \mathsf{OSp}^{\1}_{(\bbc,\bbr)}
        
        \arrow[from=5-5,to=5-3,"\mcf_c", shift left, labels=below]

        \arrow[from=5-5,to=6-5,"\mcf_0",labels=right, shift left]

         \arrow[from=5-5,to=4-5,"\mcf_A",labels=left, shift left]
        \\

        %row 6
        %6c1
        \mathsf{OSp}_{(\bbr,\bbr)}
        
        \arrow[from=6-1,to=6-3, "\textcolor{red}{\mcf_c*}",shift left,labels=above]

        \arrow[from=6-1,to=5-1,"\textcolor{violet}{\mcf_u}",labels=left, shift left]
        
        %6c3
        
        & & \mathsf{OSp}_{(\bbc,\bbc)}
        
        \arrow[from=6-3,to=6-1,"\textcolor{red}{\mcf_0*}", shift left,labels=below] 
        
        \arrow[from=6-3, to=6-5, "\mcf_0", shift left, labels=above]

        \arrow[from=6-3,to=5-3,"\textcolor{violet}{\mcf_u}",labels=left, shift left]
        
        %6c5
        & & \mathsf{OSp}_{(\bbc,\bbr)}
        
        \arrow[from=6-5,to=6-3,"\mcf_c", shift left, labels=below]

        \arrow[from=6-5,to=5-5,"\textcolor{violet}{\mcf_u}",labels=left, shift left]
    \end{tikzcd}
\end{center}

 The present paper and the recent paper \cite[Section~4]{bck23} contain the work proving Theorem~\ref{main}. In \cite{bck23}, the author (along with their coauthors) discussed the interaction between real $G$-spaces and their complexifications, and proved some relations between the injectivity of the real and the complex categories. In fact, complexification, denoted $\mcf_c$, is functorial and in \cite{bck23}, we showed that $\mcf_c(I(X))=I(\mcf_c(X))$ for a real operator space/system $X$ (resp.\ a real $G$-operator space/system, where $G$ is finite or compact). Moreover, all of the functors which are labeled with an asterisk are proved in \cite[Theorem~4.2, Proposition~4.9, Theorem~4.10, Corollary~4.13, and Theorem~4.16]{bck23}.

This paper is concerned with all of the functors in the diagram not discussed in the previous paper. Section~\ref{s2} will define the functors that are used in the above theorem/diagram and that will be referenced throughout the paper. Further, it will build up basic theory of injectivity for the ``intermediate categories'', which has not been explored in the literature yet. Theorem~\ref{main} will be proved in sections \ref{s3} and \ref{s4}, where Section~\ref{s3} focuses on all of the forgetful functors in the diagram (for which there are many) and where Section~\ref{s4} focuses on all of the rest of the functors. Section~\ref{s5} will explore properties of these functors and Section~\ref{s6} discusses the difference between the standard definition of injectivity and the purely categorical definition of injectivity. Appendix~\ref{appendix} by David Blecher investigates the unitization functor $\mcf_u$ and its properties.

Let us fix some terminology and notation. A \textit{complex operator space} $X$ (or a $\bbc$-operator space) is a subspace of $B(\mch)$ for a complex Hilbert space $\mch$. If $X$ contains the unit, then $X$ is called a \textit{unital complex operator space}. If $X$ is an operator space which is unital and self-adjoint, then $X$ is called a \textit{complex operator system}. A \textit{real operator space} (or an $\bbr$-operator space) is a real subspace of $B(\mch)$ for a real Hilbert space $\mch$. If $X$ contains the unit and is self-adjoint, then $X$ is a \textit{real operator system} (\cite{ruan03a}). 

An \textit{operator space complexification} of a real operator space $X$ is a pair $(X_c,\kappa)$ consisting of a complex operator space $X_c$ and a real linear complete isometry $\kappa:X\to X_c$ such that $X_c=\kappa(X)\oplus\ii\,\kappa(X)$ as a vector space. Usually $\kappa(X)$ is identified with $X$ and write $X_c=X+\ii\, X$. An operator space complexification $X_c=X+\ii\, X$ of a real operator space $X$ is said to be \textit{reasonable} if the map $\theta_X:x+\ii y\mapsto x-\ii y$ is a complete isometry for $x,y\in X$ (\cite{ruan03b}).

Let $X$ be a real operator space and let $\{\norm{\cdot}_n\}$ be the canonical complex operator space norm on $X_c$ given by the identification $$X_c=X+\ii X\cong\left\{ \begin{bmatrix}
    x & -y \\
    y & x
\end{bmatrix}: x,y\in X \right\}\subseteq M_2(X).$$ The following theorem shows that up to complete isometry, $X_c$ is the only reasonable complex extension of the original operator space matrix norm on $X$.

\begin{thm}[{\cite[Theorem~3.1]{ruan03b}}]\label{ruan}
    Let $X$ be a real operator space. If a complex operator space matrix norm $\{|\lVert\,\cdot\,|\rVert_n\}$ on $X_c$ is a reasonable complex extension of the original matrix norm on $X$, then $\{|\lVert\,\cdot\,|\rVert_n\}$ must be equal to the canonical matrix norm $\{\norm{\,\cdot\,}\}$ on $X_c$.
\end{thm}
A unital map between complex or real operator systems is completely positive if and only
if it is completely contractive. If $T:X\to Y$ is a completely contractive (resp.\ unital completely contractive, unital completely positive) $\bbr$-linear map between real operator spaces (resp.\ unital operator spaces, operator systems), then $T_c:X_c\to Y_c$, defined by $x+\ii y\mapsto T(x)+\ii T(y)$ for $x,y\in X$, is a completely contractive (resp.\ unital completely contractive, unital completely positive) $\bbc$-linear map.

For a given group $G$, a \textit{$G$-set} refers to a given action on a set $X$, denoted by $G \acts X$. Given two
$G$-sets $X$ and $Y$ , a \textit{$G$-equivariant map} (or \textit{$G$-map}) is a map $f : X \to Y$ such that $$f(gx)=gf(x),\quad\forall g\in G\text{ and }x\in X.$$ Going forward, let $G$ be a discrete group. A \textit{complex $G$-operator space} (resp.\ \textit{real $G$-operator space}) is a complex (resp.\ real) operator space $X$ such that $G\acts X$ by $\bbc$-linear (resp.\ $\bbr$-linear) surjective complete isometries. If $X$ is unital and the $G\acts X$ by unital surjective complete isometries, then $X$ is called a \textit{unital $G$-operator space}. A \textit{complex $G$-operator system} (resp.\ \textit{real $G$-operator system}) is a complex (resp.\ real) operator system such that $G\acts X$ by unital complete order $\bbc$-linear (resp.\ $\bbr$-linear) isomorphisms. Note that Hamana’s
complex G-modules in \cite{hamana85} are the complex G-operator systems mentioned above.

\begin{lemma}[{\cite[Lemma~4.4]{bck23}}]
    If $X$ is a real $G$-operator space (resp.\ real $G$-operator system), then $X_c$ with action $g(x + \ii y) = gx + \ii gy$ is a complex $G$-operator space (resp.\ complex $G$-operator system), $\theta_X$ is $G$-equivariant, and the canonical projection $X_c\to X$ is a $G$-equivariant complete contractive (resp.\ unital completely positive) $\bbr$-linear map.
\end{lemma}

 In \cite{hamana85}, Hamana says a complex $G$-operator system $Z$ is \textit{$G$-injective} if for every $G$-equivariant unital completely isometric map $\kappa:X\to Y$ and every $G$-equivariant unital completely positive map $\vphi:X\to Z$, there is a $G$-equivariant unital completely positive map $\tilde{\vphi}:Y\to Z$ such that $\tilde{\vphi}\circ\kappa=\vphi$. This is simply being injective in the category he is considering (where complex $G$-operator systems are the objects and $G$-equivariant unital completely positive $\bbc$-linear maps are the morphisms), so we will drop the ``$G$-'' in $G$-injectivity going forward and just refer to this as being injective in this category. We will also do this for the terms ``$G$-rigid'' and ``$G$-essential''. 
 
 For $X$ in any of our categories, an \textit{extension} of $X$ is a pair $(Z,\kappa)$ consisting of another object from the category, $Z$, and a completely isometric morphism from the category, $\kappa:X\to Z$. An extension $(Z,\kappa)$ of $X$ is:
 \begin{itemize}
    \item \textit{injective} if $Z$ is an injective object in that category;
    
    \item \textit{rigid} if whenever $T:Z\to Z$ is another morphism in that category such that $T\circ\kappa=\kappa$, then $T=\textup{Id}_Z$;

    \item \textit{rigid} if whenever $T:Z\to W$ is a morphism such that $T\circ\kappa$ is a completely isometric morphism, then $T$ is a completely isometric morphism.
 \end{itemize}

An \textit{injective envelope} of an object $X$ is an extension $(Z,\kappa)$ in its category which is injective and essential. In these categories, an injective extension is a rigid extension if and only if it is an essential extension \cite{hamana79b,hamana85,sharma14,ahcthesis}.

\begin{thm}[{\cite[Theorem~4.5]{bck23}}]
    Every real or complex $G$-operator space (resp.\ $G$-operator system) $X$ has a real or complex injective envelope, written $(Z,\kappa)$ (in the category of real or complex G-operator spaces (resp.\ G-operator systems). Here $\kappa:X\to Z$ is a real or complex linear $G$-equivariant complete isometry (resp.\ unital complete order embedding). This injective envelope is unique in the sense that for any other injective envelope $(W,\iota)$ of X in the categories above, there is a $G$-equivariant isomorphism $\psi:Z\to W$ (that is, $\psi$ and $\psi^{-1}$ are morphisms in the category) with $\psi\circ\kappa=\iota$.
\end{thm}

For any $X\subseteq B(\mch)$, we may define a $G$-action on $\ell^{\infty}(G,X)$
and $\ell^{\infty}(G, B(\mch))$ by $(tf)(s) = f(t^{-1}s)$ for $s, t\in G$. There is also an involution on
$\ell^{\infty}(G, B(\mch)): f^*(s) = f(s)^*$. For a $G$-operator space (resp. $G$-operator system) $X\subseteq B(\mch)$, we have the $G$-equivariant inclusions
$X \subseteq \ell^{\infty}(G, X) \subseteq \ell^{\infty}(G, B(H))$,
where the first of these is the map $j(x)(s) = s
^{-1}x$ for $s \in G, x \in X$. One can check that $j:X\to \ell^{\infty}(G,X)$ is a $G$-equivariant complete isometry (resp.\ completely contractive, unital completely contractive, unital completely positive) map and that $\ell^{\infty}(G,X)$ is injective as a real or complex $G$-operator space (resp.\ $G$-operator system) and a real or complex operator space (resp.\ $G$-operator system) if and only if $X$ is an injective real or complex operator space (resp.\ operator system).

\begin{lemma}[{\cite[Lemma~4.7]{bck23}}]
    Let $X$ be a real or complex $G$-operator space (resp.\ real or complex $G$-operator system). Then $X$ is injective in the category of real or complex $G$-operator systems and $G$-equivariant $\bbc$-linear or $\bbr$-linear morphisms if and only if $X$ is injective in the category consisting of  real or complex operator spaces (resp.\ real or complex operator system) and there is a $G$-equivariant morphism $\psi : \ell^{\infty}(G, X) \to X$ such that $\phi\circ j = \textup{Id}_X$.
\end{lemma}

\section*{Acknowledgement}
The author would like to give many thanks to David P. Blecher and Mehrdad Kalantar, their PhD advisors, for all of the input, time, effort, patience, and support they have given while working on this project. The author would also like to thank Caleb Barnett for his helpful comments and suggestions.

%-----New Section-----%

\section{Categories and functors to be considered}\label{s2}

%-----Subsection-----%

\subsection{Category and functor notation}

Throughout the paper, $\bbf_i$ is either $\bbr$ or $\bbc$ for $i=1,2$. The notation $\mathsf{OSp}_{(\bbf_1,\bbf_2)}$ will be used to denote the category whose objects are $\bbf_1$-operator spaces and whose morphisms are completely contractive $\bbf_2$-linear maps. The notation $\mathsf{OSy}_{(\bbf_1,\bbf_2)}$ will be used to denote the category whose objects are $\bbf_1$-operator systems and whose morphisms are unital completely positive $\bbf_2$-linear maps. If the notation additionally has a superscript of ``$G$'' (resp.\ ``$\1$''), then the objects are $G$-operator spaces/systems (resp.\ unital operator spaces) and the morphisms are additionally $G$-equivariant (resp.\ unital). Note that the superscript ``$\1$'' is reserved for operator spaces specifically, since operator systems are unital by definition.

The categories with ``$(\bbc,\bbc)$" (resp.\ $(\bbr,\bbr)$,\,$(\bbc,\bbr)$) will be called complex categories (resp.\ real categories, intermediate categories), injectivity in these category may be referred to as complex injectivity (resp.\ real injectivity, intermediate injectivity), and an injective envelope of an object in these categories may be referred to as a complex injective envelope (resp.\ a real injective envelope, an intermediate injective envelope). The case $(\bbr,\bbc)$ will never be considered for any category.

The different functors that will be considered are
\begin{enumerate}
    \item the forgetful functor $\mcf_0$,
    \item the trivial $G$-action functor $\mcf_G$,
    \item the Arveson functor $\mcf_A$,
    \item the complexification functor $\mcf_c$. 
\end{enumerate}

The forgetful functor $\mcf_0$ could have various combinations of domain and co-domain, but the structure that is forgotten should be clear. For example, $$\mcf_0:\mathsf{OSy}_{(\bbc,\bbc)}\to\mathsf{OSy}_{(\bbr,\bbr)}$$ forgets the complex structure on the objects and on the morphisms. That is, a complex operator system is mapped to itself, but considered only as a real operator system and a unital completely positive $\bbc$-linear map is mapped to itself, but considered only as a unital completely positive $\bbr$-linear map.

The trivial $G$-action functor $\mcf_G$ is the functor which takes an $\bbf_1$-operator space (resp.\ $\bbf_1$-unital operator space, $\bbf_1$-operator system) $X$ to itself, but now with $G$ acting on $X$ trivially, and which takes a completely contractive (resp.\ unital completely contractive, unital completely positive) $\bbf_2$-linear map $u:X\to Y$ to itself as a trivially $G$-equivariant completely contractive (resp.\ unital completely contractive, unital completely positive) $\bbf_2$-linear map.

The Arveson functor $\mcf_{A}$ is the functor that maps a unital $\bbf_1$-operator space $X$ to the canonical (essentially unique) $\bbf_1$-operator system $X+X^{\star}$ (\cite[1.3.7]{blm04}) and maps a unital completely contractive $\bbf_2$-linear map $u:X\to \textup{Ran}(u)$ to the canonical unital completely positive $\bbf_2$-linear map $\tilde{u}:X+X^{\star}\to \textup{Ran}(u)+\textup{Ran}(u)^{\star}$, where $\tilde{u}(x+y^*)=u(x)+u(y)^*$ \cite[Lemma~1.3.6]{blm04}. In the case that a unital operator space is a unital $G$-operator space, there is a canonical $G$-action on $\mcf_A(X)$ given by $g(x+y^*)=gx+(gy)^*$. Since $X\cap X^{\star}$ is an operator system and any unital completely contractive map on an operator system is $^*$-linear, this action is well-defined and turns $\mcf_A(X)$ into a $G$-operator system.

The complexification functor $\mcf_{c}$ is the functor that takes a real operator space (resp.\ unital operator space, operator system) $X$ and maps it to the unique reasonable complexification $X+\ii X$ as in Theorem~\ref{ruan} and takes an $\bbr$-linear morphism $T:X\to Y$ and maps it to it is complexification $T_c:X+\ii X\to Y+\ii Y$, where $T_c(x+\ii y)=T(x)+\ii T(y)$ for $x,y\in X$. 

%-----Subsection-----%
\subsection{The intermediate categories}

\begin{defn}
	A complex operator space $Z$ (resp.\ unital complex operator space, complex operator system) is injective in the category $\mathsf{OSp}_{(\bbc,\bbr)}$ (resp.\ $\mathsf{OSp}^{\1}_{(\bbc,\bbr)}$, $\mathsf{OSy}_{(\bbc,\bbr)}$) if for any complex operator spaces (resp.\ unital complex operator spaces, complex operator systems) $X$ and $Y$ with a complete isometry (resp.\ unital complete isometry) $\kappa:X\to Y$, any completely contractive (resp.\ unital completely contractive, unital completely positive) $\bbr$-linear map $\phi:X\to Z$, there exists a completely contractive (resp.\ unital completely contractive, unital completely positive) $\bbr$-linear extension $\tilde{\phi}:Y\to Z$ such that $\norm{\tilde{\phi}}_{cb}=\norm{\phi}_{cb}$.
\end{defn}

\begin{thm}\label{bhinj}
	Let $\mch$ be a complex Hilbert space. Then $B(\mch)$ is injective in $\mathsf{OSp}_{(\bbc,\bbr)}$ (resp.\ $\mathsf{OSp}^{\1}_{(\bbc,\bbr)}$, $\mathsf{OSy}_{(\bbc,\bbr)}$).
\end{thm}

\begin{proof}
	Let $\kappa:X\to Y$ be an $\bbr$-linear complete isometry (resp.\ unital complete isometry) and $\vphi:X\to B(\mch)$ be a completely contractive (resp.\ unital completely contractive, unital completely positive) $\bbr$-linear map. Consider $X$, $Y$, and $B(\mch)$ as real operator spaces (resp.\ unital real operator spaces, real operator systems). By \cite[Lemma~4.1]{bck23}, $B(\mch)$ is injective in $\mathsf{OSp}_{(\bbr,\bbr)}$ (resp.\ $\mathsf{OSp}^{\1}_{(\bbr,\bbr)}$, $\mathsf{OSy}_{(\bbr,\bbr)}$). Hence, the desired extension exists.
\end{proof}

\begin{cor}\label{intermedinjiff}
	A complex operator space (resp.\ unital complex operator space, complex operator system) $Z\subset B(\mch)$ for a complex Hilbert space $\mch$ is injective in $\mathsf{OSp}_{(\bbc,\bbr)}$ (resp.\ $\mathsf{OSp}^{\1}_{(\bbc,\bbr)}$, $\mathsf{OSy}_{(\bbc,\bbr)}$) if and only if it is the range of a completely contractive (resp.\ unital completely contractive, unital completely positive) idempotent $\bbr$-linear map from $B(\mch)$ onto $Z$.
\end{cor}

Let $X$ and $Y$ be complex operator spaces (resp.\ unital complex operator spaces, complex operator systems) and let the map $\kappa:X\to Y$ be an $\bbr$-linear complete  isometry (resp.\ unital complete  isometry). Then the pair $(Y,\kappa)$ is called an \textit{extension of} $X$ in $\mathsf{OSp}_{(\bbc,\bbr)}$ (resp.\ $\mathsf{OSp}^{\1}_{(\bbc,\bbr)}$, $\mathsf{OSy}_{(\bbc,\bbr)}$). 

\begin{defn}
    If $(Y,\kappa)$ is an extension of the complex operator space (resp.\ unital complex operator space, complex operator system) $X$ in $\mathsf{OSp}_{(\bbc,\bbr)}$ (resp.\ $\mathsf{OSp}^{\1}_{(\bbc,\bbr)}$, $\mathsf{OSy}_{(\bbc,\bbr)}$), then $(Y,\kappa)$ is a \textit{rigid extension} of $X$ if for any complete contraction (resp.\ unital complete contraction, unital completely positive) $\vphi:Y\to Y$,  $\vphi\circ\kappa=\kappa$ implies $\vphi=\textup{Id}_Y$. The extension $(Y,\kappa)$ is said to be an \textit{essential extension of} $X$ if whenever $u:Y\to Z$ is a completely contractive (resp.\ unital completely contractive, unital completely positive) $\bbr$-linear map for some complex operator space (resp.\ unital complex operator space, complex operator system) $Z$ such that $u\circ\kappa$ is a (resp.\ unital) $\bbr$-linear complete isometry, then $u$ is a (resp.\ unital) $\bbr$-linear complete isometry. The extension $(Y,\kappa)$ is a \textit{injective envelope of} $X$ in $\mathsf{OSp}_{(\bbc,\bbr)}$ (resp.\ $\mathsf{OSp}^{\1}_{(\bbc,\bbr)}$, $\mathsf{OSy}_{(\bbc,\bbr)}$) provided that $(Y,\kappa)$ is an injective and an essential extension.
 \end{defn}
 
A similar construction to Hamana's in \cite{hamana79b} can be used to prove that any complex operator space (resp.\ unital operator space, operator system) has an injective envelope in the intermediate category.

\begin{thm}
	Any complex operator space (resp.\ unital complex operator spaces, complex operator systems) $X$ has an injective envelope in $\mathsf{OSp}_{(\bbc,\bbr)}$ (resp.\ $\mathsf{OSp}^{\1}_{(\bbc,\bbr)}$, $\mathsf{OSy}_{(\bbc,\bbr)}$). It is unique in the sense that for any two injective envelopes of $X$ say $(Z,\iota)$ and $(Z',\iota')$, there is an $\bbr$-linear complete isometry (resp.\ unital complete isometry) $\psi:Z\to Z'$ with $\psi\circ\iota=\iota'$.
\end{thm}

\begin{remark}
        $(1)$ Every real injective envelope $(Z,\kappa)$ of a complex operator space $X$ can be made into an intermediate injective envelope by giving $Z$ a complex operator space structure (see \cite[Section~3]{b23}). This is done by noting that if $(W,j)$ is a complex injective envelope of $X$, then it also is a real injective envelope of $X$ and there exists a bijective completely isometric map $\psi:W\to Z$ such that $\psi^{-1}$ is completely isometric. Then, define complex multiplication on $Z$ by $\ii z \coloneqq \psi(\ii \psi^{-1}(z))$ for all $z\in Z$. Letting $\tilde{Z}=Z$ as a complex operator space, it follows that $W\cong \tilde{Z}$ as complex operator spaces (since $\psi$ and $\psi^{-1}$ are $\bbc$-linear using this new multiplication). Hence, $(\tilde{Z},\kappa)$ is an intermediate injective envelope of $X$.
        \\[0.8ex]
        $(2)$ Every real injective envelope $(Z,\kappa)$ of a complex unital operator space or complex operator system $X$ can be made into an intermediate injective envelope for $X$. To do this, take the real injective envelope $(Z,\kappa)$ and consider $Z$ as a real operator space and make the same argument as above to consider it as a complex operator space. Then, appeal to \cite[Theorem~4.14]{bck23} to get the result.
        \\[0.8ex]
        $(3)$ Intermediate injective envelopes of a complex operator space (resp. unital operator space, operator system) $X$ are essentially the same as the real injective envelopes of $X$ which are also complex operator spaces (resp. unital operator spaces, operator systems)
    
\end{remark}

%-----Subsection-----%
\subsection{The $G$-categories}

\begin{defn}
	An $\bbf_1$-$G$-operator space (resp.\ unital $\bbf_1$-$G$-operator space, $\bbf_1$-$G$-operator system) $Z$ is injective in the category $\mathsf{OSp}^{G}_{(\bbf_1,\bbf_2)}$ (resp.\ $\mathsf{OSp}^{\1,G}_{(\bbf_1,\bbf_2)}$, $\mathsf{OSy}^{G}_{(\bbf_1,\bbf_2)}$) if for any  $\bbf_1$-$G$-operator spaces (resp.\ unital $\bbf_1$-$G$-operator spaces, $\bbf_1$-$G$-operator systems) $X$ and $Y$ with a $G$-equivariant $\bbf_2$-liner complete isometry (resp.\ unital complete isometry) $\kappa:X\to Y$, and any  $G$-equivariant completely contractive (resp.\ unital completely contractive, unital completely positive) $\bbf_2$-linear map $\vphi:X\to Z$, there exists a $G$-equivariant completely contractive (resp.\ unital completely contractive, unital completely positive) $\bbf_2$-linear map $\tilde{\vphi}:Y\to Z$ such that $\tilde{\vphi}\circ\kappa=\vphi$ and $\norm{\tilde{\vphi}}_{cb}=\norm{\vphi}_{cb}$. 
\end{defn}

Let $X$ and $Z$ be  $\bbf_1$-$G$-operator spaces (unital $\bbf_1$-$G$-operator spaces, $\bbf_1$-$G$-operator systems) and let the map $\kappa:X\to Z$ be a $G$-equivariant $\bbf_2$-linear complete isometry (resp.\ unital complete isometry). Then the pair $(Z,\kappa)$ is called an \textit{extension} of $X$ in $\mathsf{OSp}^{G}_{(\bbf_1,\bbf_2)}$ (resp.\ $\mathsf{OSp}^{\1,G}_{(\bbf_1,\bbf_2)}$, $\mathsf{OSy}^{G}_{(\bbf_1,\bbf_2)}$).

 \begin{defn}
    The extension $(Z,\kappa)$ of $X$ in $\mathsf{OSp}^{G}_{(\bbf_1,\bbf_2)}$ (resp.\ $\mathsf{OSp}^{\1,G}_{(\bbf_1,\bbf_2)}$, $\mathsf{OSy}^{G}_{(\bbf_1,\bbf_2)}$) is a \textit{rigid extension} of $X$ 
    if whenever $\vphi:Z\to Z$ is a $G$-equivariant completely contractive (resp.\ unital completely contractive, unital completely positive) $\bbf_2$-linear map such that $\vphi\circ\kappa=\kappa$, then $\vphi=\textup{Id}_Z$. The extension $(Z,\kappa)$ is said to be an \textit{essential extension of} $X$, if whenever $\vphi:Z\to W$ is a $G$-equivariant completely contractive (resp.\ unital completely contractive, unital completely positive) $\bbf_2$-linear map, for some $\bbf_1$-$G$-operator space (resp.\ unital $\bbf_1$-$G$-operator space, $\bbf_1$-$G$-operator system) $W$, such that $\vphi\circ\kappa$ is a $G$-equivariant $\bbf_2$-linear complete isometry (resp.\ unital complete isometry), then $\vphi$ is a $G$-equivariant $\bbr$-linear complete isometry (resp.\ unital complete isometry).
\end{defn}

\begin{lemma}\label{crellgvginj}
	Let $G$ be a discrete group. If $X\in\mathsf{OSp}_{(\bbf_1,\bbf_2)}$ (resp.\ $\mathsf{OSp}^{\1}_{(\bbf_1,\bbf_2)}$, $\mathsf{OSy}_{(\bbf_1,\bbf_2}$) is injective, then the $\bbf_1$-$G$-operator space (resp.\ unital $\bbf_1$-$G$-operator space, $\bbf_1$-$G$-operator system) $\ell^{\infty}(G,X)$ is injective in $\mathsf{OSp}^{G}_{(\bbf_1,\bbf_2)}$ (resp.\ $\mathsf{OSp}^{\1,G}_{(\bbf_1,\bbf_2)}$, $\mathsf{OSy}^{G}_{(\bbf_1,\bbf_2)}$) and in $\mathsf{OSp}_{(\bbf_1,\bbf_2)}$ (resp.\ $\mathsf{OSp}^{\1}_{(\bbf_1,\bbf_2)}$, $\mathsf{OSy}_{(\bbf_1,\bbf_2)}$).     In particular, $\ell^{\infty}(G,B(\mch))$ is injective in both $\mathsf{OSp}^{G}_{(\bbf_1,\bbf_2)}$ (resp.\ $\mathsf{OSp}^{\1,G}_{(\bbf_1,\bbf_2)}$, $\mathsf{OSy}^{G}_{(\bbf_1,\bbf_2)}$) and $\mathsf{OSp}_{(\bbf_1,\bbf_2)}$ (resp.\ $\mathsf{OSp}^{\1}_{(\bbf_1,\bbf_2)}$, $\mathsf{OSy}_{(\bbf_1,\bbf_2)}$).
\end{lemma}

\begin{proof}
	The proof is the same as in \cite[Lemma~2.2]{hamana85}, considering the appropriate objects and morphisms.
\end{proof}

Similarly as in \cite[Remark~2.3]{hamana85}, any $\bbf_1$-$G$-operator space (resp.\ unital $G$-operator space, $G$-operator system) $X\subseteq B(\mch)$, the map $j:X\to\ell^{\infty}(G,B(\mch))$, $j(x)(t)=t^{-1} x$ for $x\in X\text{ and }t\in G$, is a $G$-equivariant completely isometric (resp.\ unital completely isometric) $\bbf_2$-linear map with $j(X)\subseteq\ell^{\infty}(G,X)\subseteq\ell^{\infty}(G,B(\mch))$. The next lemma follows as a consequence of this and is essential in the proof of existence of an injective envelope.

\begin{lemma}
    Each $\bbf_1$-$G$-operator space (resp.\ unital $G$-operator space, $G$-operator system) has an injective extension in $\mathsf{OSp}^{G}_{(\bbf_1,\bbf_2)}$ (resp.\ $\mathsf{OSp}^{\1,G}_{(\bbf_1,\bbf_2)}$, $\mathsf{OSy}^{G}_{(\bbf_1,\bbf_2)}$), namely $(\ell^{\infty}(G,B(\mch)),j)$. 
\end{lemma}

\begin{prop}\label{crginjiff}
	The space $X$ is injective in $\mathsf{OSp}^{G}_{(\bbf_1,\bbf_2)}$ (resp.\ $\mathsf{OSp}^{\1,G}_{(\bbf_1,\bbf_2)}$, $\mathsf{OSy}^{G}_{(\bbf_1,\bbf_2)}$) if and only if $X$ is injective in $\mathsf{OSp}_{(\bbf_1,\bbf_2)}$ (resp.\ $\mathsf{OSp}^{\1}_{(\bbf_1,\bbf_2)}$, $\mathsf{OSy}_{(\bbf_1,\bbf_2)}$) and there exists a surjective $G$-equivariant completely contractive (resp.\ unital completely contractive, unital completely positive) idempotent $\bbf_2$-linear map $\phi:\ell^{\infty}(G,X)\to j(X)$ with $\vphi\circ j=\textup{Id}_X$.
\end{prop}

\begin{proof}
    Follows similarly to \cite[Lemma~2.2 and Remark~2.3]{hamana85}.
\end{proof}

\begin{defn}\label{gopsyinjenv}
	The pair $(Z,\kappa)$ is an \textit{injective envelope of} $X$ in $\mathsf{OSp}^{G}_{(\bbf_1,\bbf_2)}$(resp.\ $\mathsf{OSp}^{\1,G}_{(\bbf_1,\bbf_2)}$, $\mathsf{OSy}^{G}_{(\bbf_1,\bbf_2)}$) if it is both an injective and an essential extension of $X$ in $\mathsf{OSp}^{G}_{(\bbf_1,\bbf_2)}$ (resp.\ $\mathsf{OSp}^{\1,G}_{(\bbf_1,\bbf_2)}$, $\mathsf{OSy}^{G}_{(\bbf_1,\bbf_2)}$).
\end{defn}

Any complex $G$-operator space (resp.\ unital $G$-operator space, $G$-operator system) has an injective envelope in the intermediate $G$-category using a similar construction to Hamana's in \cite{hamana85}.

\begin{thm}
	Every $\bbf_1$-$G$-operator space (resp.\ unital $G$-operator space, $G$-operator system) $X$ has an injective envelope in $\mathsf{OSp}^{G}_{(\bbf_1,\bbf_2)}$ (resp.\ $\mathsf{OSp}^{\1,G}_{(\bbf_1,\bbf_2)}$, $\mathsf{OSy}^{G}_{(\bbf_1,\bbf_2)}$). It is unique in the sense that for any two injective envelopes of $X$, say $(Z, \iota)$ and $(Z',\iota')$, there is a $G$-equivariant completely isometric (resp.\ unital completely isometric) $\bbf_2$-linear map $\psi: Z \rightarrow Z'$ with $\psi \circ \iota=\iota'$. 
\end{thm}

\begin{proof}
    Follows as in \cite[Theorem~2.5]{hamana85}, considering appropriate objects and morphisms.
\end{proof}

%----New Section-----%
\section{Forgetful Functors between operator categories}\label{s3}

%-----Subsection-----%

\subsection{From complex $G$-categories to intermediate $G$-categories}
Consider the functor $\mcf_0$ from the category $\mathsf{OSp}^{G}_{(\bbc,\bbc)}$ (resp.\ $\mathsf{OSp}^{\1,G}_{(\bbc,\bbc)}$, $\mathsf{OSy}^{G}_{(\bbc,\bbc)}$) to the category $\mathsf{OSp}^{G}_{(\bbc,\bbr)}$ (resp.\ $\mathsf{OSp}^{\1,G}_{(\bbc,\bbr)}$, $\mathsf{OSy}^{G}_{(\bbc,\bbr)}$).

\begin{lemma}\label{gcpltogintlem}
	If $Z\in\mathsf{OSp}^{G}_{(\bbc,\bbc)}$ (resp.\ $\mathsf{OSp}^{\1,G}_{(\bbc,\bbc)}$, $\mathsf{OSy}^{G}_{(\bbc,\bbc)}$) is injective, then $\mcf_{0}(Z)$ is injective.
\end{lemma}

\begin{proof}
	Assume a complex $G$-operator space $Z$ (resp.\ unital $G$-operator space, $G$-operator system) is injective in $\mathsf{OSp}^{G}_{(\bbc,\bbc)}$ (resp.\ $\mathsf{OSp}^{\1,G}_{(\bbc,\bbc)}$, $\mathsf{OSy}^{G}_{(\bbc,\bbc)}$). Also, let $X$ and $Y$ be complex $G$-operator spaces (resp.\ unital $G$-operator spaces, $G$-operator systems) and let $\kappa:X\to Y$ be a $G$-equivariant $\bbr$-linear complete isometry (resp.\ unital complete isometry), and let $\varphi:X\to Z$ be a $G$-equivariant completely contractive (resp.\ unital completely contractive, unital completely positive) $\bbr$-linear map. By \cite[Corollary~4.12]{bck23}, $Z$ is equivalently injective in $\mathsf{OSy}^{G}_{(\bbr,\bbr)}$. Therefore, considering $X$, $Y$, and $Z$ as real $G$-operator systems, there exists a $G$-equivariant $\bbr$-linear extension $\tilde{\varphi}:Y\to Z$ such that $\tilde{\varphi}\circ\kappa=\varphi$ and $\norm{\tilde{\varphi}}_{cb}=\norm{\varphi}_{cb}$. Thus $Z$ is injective in $\mathsf{OSp}^{G}_{(\bbc,\bbr)}$ (resp.\ $\mathsf{OSp}^{\1,G}_{(\bbc,\bbr)}$, $\mathsf{OSy}^{G}_{(\bbc,\bbr)}$).
\end{proof}

\begin{thm}\label{gcpltogintthm}
	Let $X\in\mathsf{OSp}^{G}_{(\bbc,\bbc)}$ (resp.\ $\mathsf{OSp}^{\1,G}_{(\bbc,\bbc)}$, $\mathsf{OSy}^{G}_{(\bbc,\bbc)}$).  If $(Z,\kappa)$ is an injective envelope of $X$, then $(\mcf_0(Z),\mcf_0(\kappa))$ is an injective envelope of $\mcf_0(X)$.
\end{thm}

\begin{proof}
	Let $(Z,\kappa)$ be an injective envelope of $X$ in $\mathsf{OSp}^{G}_{(\bbc,\bbc)}$ (resp.\ $\mathsf{OSp}^{\1,G}_{(\bbc,\bbc)}$, $\mathsf{OSy}^{G}_{(\bbc,\bbc)}$). Then $Z$ is injective in $\mathsf{OSp}^{G}_{(\bbc,\bbr)}$ (resp.\ $\mathsf{OSp}^{\1,G}_{(\bbc,\bbr)}$, $\mathsf{OSy}^{G}_{(\bbc,\bbr)}$) by the result above. To show that $(Z,\kappa)$ is an injective envelope of $X$ in $\mathsf{OSp}^{G}_{(\bbc,\bbr)}$ (resp.\ $\mathsf{OSp}^{\1,G}_{(\bbc,\bbr)}$, $\mathsf{OSy}^{G}_{(\bbc,\bbr)}$), it is enough to show that $(Z,\kappa)$ is a rigid extension of $X$ in $\mathsf{OSp}^{G}_{(\bbc,\bbr)}$ (resp.\ $\mathsf{OSp}^{\1,G}_{(\bbc,\bbr)}$, $\mathsf{OSy}^{G}_{(\bbc,\bbr)}$).
    Consider $\kappa:X\to Z$ as a $\bbr$-linear map and let $T:Z\to Z$ be a $G$-equivariant completely contractive $\bbr$-linear map such that $T\circ\kappa=\kappa$. Now, consider $X$ and $Z$ as real operator spaces, where $(Z,\kappa)$ is also an injective envelope of $X$ in $\mathsf{OSp}^{G}_{(\bbr,\bbr)}$ (resp.\ $\mathsf{OSp}^{\1,G}_{(\bbr,\bbr)}$,$\mathsf{OSy}^{G}_{(\bbr,\bbr)}$). Then $(Z,\kappa)$ is equivalently a rigid extension of $X$ in $\mathsf{OSp}^{G}_{(\bbr,\bbr)}$ (resp.\ $\mathsf{OSp}^{\1,G}_{(\bbr,\bbr)}$,$\mathsf{OSy}^{G}_{(\bbr,\bbr)}$). Hence the $\bbr$-linear map $T$ must be the identity on $Z$. But this implies that $(Z,\kappa)$ is a rigid extension of $X$ in $\mathsf{OSp}^{G}_{(\bbc,\bbr)}$ (resp.\ $\mathsf{OSp}^{\1,G}_{(\bbc,\bbr)}$, $\mathsf{OSy}^{G}_{(\bbc,\bbr)}$). Thus, equivalently, $(Z,\kappa)$ is an injective envelope of $X$ in $\mathsf{OSp}^{G}_{(\bbc,\bbr)}$ (resp.\ $\mathsf{OSp}^{\1,G}_{(\bbc,\bbr)}$, $\mathsf{OSy}^{G}_{(\bbc,\bbr)}$).
\end{proof}

\begin{remark}
    For the functors going from the complex categories to the intermediate categories, variants of Lemma~\ref{gcpltogintlem} and Theorem~\ref{gcpltogintthm} also hold and follow directly from them, taking $G=\{e\}$. Alternatively, one can prove the results directly, considering the appropriate objects and morphisms. 
\end{remark}

%-----Subection-----%

\subsection{From $G$-operator system categories to unital $G$-operator space categories} Consider the functor $\mcf_0$ from the category $\mathsf{OSy}^{G}_{(\bbf_1,\bbf_2)}$ to the category $\mathsf{OSp}^{\1,G}_{(\bbf_1,\bbf_2)}$.

\begin{lemma}\label{ffinj}
	Let $X\in\mathsf{OSy}^{G}_{(\bbf_1,\bbf_2)}$ be injective. Then $\mcf_0(X)$ is injective.
\end{lemma}

\begin{proof}
	Assume $X$ is injective in $\mathsf{OSy}^{G}_{(\bbf_1,\bbf_2)}$. Equivalently, $X$ is injective in $\mathsf{OSy}_{(\bbf_1,\bbf_2)}$ and there exists a $G$-equivariant unital completely positive idempotent $\bbf_2$-linear map $T$ from $\ell^{\infty}(G,X)$ onto $X$. But $T$ is equivalently a surjective $G$-equivariant unital completely contractive idempotent $\bbf_2$-linear map. So, it suffices to show that $X$ is also injective in $\mathsf{OSp}^{\1,G}_{(\bbf_1,\bbf_2)}$ when considering $X$ as a unital operator space. But $X\subset B(\mch)$ for some $\bbf_2$-Hilbert space $\mch$ and since $X$ is injective in $\mathsf{OSy}_{(\bbf_1,\bbf_2)}$, there exists a unital completely positive idempotent $\bbf_2$-linear map from $B(\mch)$ onto $X$. Equivalently, there exists a unital completely contractive idempotent $\bbf_2$-linear map from $B(\mch)$ onto $X$, and therefore $X$ is injective in $\mathsf{OSp}^{\1}_{(\bbf_1,\bbf_2)}$. Thus, $X$ is injective in $\mathsf{OSp}^{\1,G}_{(\bbf_1,\bbf_2)}$.
\end{proof}

\begin{thm}\label{gosytougosp}
	Let $X\in\mathsf{OSy}^{G}_{(\bbf_1,\bbf_2)}$. If $(Z,\kappa)$ is an injective envelope of $X$, then $(\mcf_0(Z),\mcf_0(\kappa))$ is an injective envelope of $\mcf_0(X)$.
\end{thm}

\begin{proof}
 Let $(Z,\kappa)$ be an injective envelope of $X$ in  $\mathsf{OSy}^{G}_{(\bbf_1,\bbf_2)}$. By Lemma~\ref{ffinj}, $Z$ is injective in $\mathsf{OSp}^{\1,G}_{(\bbf_1,\bbf_2)}$. To show that $(Z,\kappa)$ is an injective envelope of $X$ in $\mathsf{OSp}^{\1,G}_{(\bbf_1,\bbf_2)}$, it suffices to show that it is a rigid extension in $\mathsf{OSp}^{\1,G}_{(\bbf_1,\bbf_2)}$. Let $T:Z\to Z$ be a $G$-equivariant unital completely contractive $\bbf_2$-linear map such that $T\circ\kappa=\kappa$. It is equivalently a $G$-equivariant unital completely positive $\bbf_2$-linear map such that $T\circ\kappa=\kappa$. Therefore, by the rigidity of $(Z,\kappa)$ in $\mathsf{OSy}^{G}_{(\bbf_1,\bbf_2)}$,  $T=\textup{Id}_{Z}$. Thus, $(Z,\kappa)$ is a rigid extension of $X$ in $\mathsf{OSp}^{\1,G}_{(\bbf_1,\bbf_2)}$, and hence an injective envelope of $X$ in $\mathsf{OSp}^{\1,G}_{(\bbf_1,\bbf_2)}.$
\end{proof}

\begin{remark}
    For the functors going from the operator system categories to the unital operator space categories, variants of Lemma~\ref{ffinj} and Theorem~\ref{gosytougosp} also hold and follow directly from them, taking $G=\{e\}$. Alternatively, one can prove the results directly, considering the appropriate objects and morphisms. 
\end{remark}

%-----Subsection-----%

\subsection{From unital $G$-operator space categories to $G$-operator space categories} Consider the functor $\mcf_0$ from the category $\mathsf{OSp}^{\1,G}_{(\bbf_1,\bbf_2)}$ to the category $\mathsf{OSp}^{G}_{(\bbf_1,\bbf_2)}$.

\begin{lemma}\label{ffinj2}
	Let $X\in\mathsf{OSp}^{\1,G}_{(\bbf_1,\bbf_2)}$ be injective. Then $\mcf_0(X)$ is injective.
\end{lemma}

\begin{proof}
	Assume $X$ is injective in $\mathsf{OSp}^{\1,G}_{(\bbf_1,\bbf_2)}$. Equivalently, $X$ is injective in $\mathsf{OSp}^{\1}_{(\bbf_1,\bbf_2)}$ and there exists a surjective $G$-equivariant unital completely contractive idempotent $\bbf_2$-linear map $T$ from $\ell^{\infty}(G,X)$ onto $X$. It follows that $T$ is equivalently a surjective $G$-equivariant completely contractive idempotent $\bbf_2$-linear map. Since $X$ is injective in $\mathsf{OSp}^{\1}_{(\bbf_1,\bbf_2)}$ and $X\subset B(\mch)$ for some $\bbf_1$-Hilbert space, there exists a unital completely contractive idempotent $\bbf_2$-linear map from $B(\mch)$ onto $X$. Equivalently there exists a completely contractive idempotent $\bbf_2$-linear map from $B(\mch)$ onto $X$, hence $X$ is injective in $\mathsf{OSp}_{(\bbf_1,\bbf_2)}$. 
    Thus $X$ is equivalently injective in $\mathsf{OSp}^{G}_{(\bbf_1,\bbf_2)}$.
\end{proof}

\begin{thm}\label{ugopstogopsp}
	Let $X\in\mathsf{OSp}^{\1,G}_{(\bbf_1,\bbf_2)}$ and let $(Z,\kappa)$ be an injective envelope of $X$. Then $(\mcf_0(Z),\mcf_0(\kappa))$ is an injective envelope of $\mcf_0(X)$.
\end{thm}

\begin{proof}
Let $(Z,\kappa)$ be an injective envelope of $X$ in  $\mathsf{OSp}^{\1,G}_{(\bbf_1,\bbf_2)}$. By Lemma~\ref{ffinj2}, $Z$ is injective in $\mathsf{OSp}^{G}_{(\bbf_1,\bbf_2)}$. To show that $(Z,\kappa)$ is an injective envelope of $X$ in $\mathsf{OSp}^{G}_{(\bbf_1,\bbf_2)}$, it suffices to show that it is a rigid extension of $X$ in $\mathsf{OSp}^{G}_{(\bbf_1,\bbf_2)}$. Let $T:Z\to Z$ be a $G$-equivariant completely contractive $\bbf_2$-linear map such that $T\circ\kappa=\kappa$. Since $X$ and $Z$ are unital and $\kappa$ is unital, $T$ must be unital. Therefore, by the rigidity of $(Z,\kappa)$ in $\mathsf{OSp}^{\1,G}_{(\bbf_1,\bbf_2)}$, $T=\textup{Id}_{Z}$. Thus, $(Z,\kappa)$ is a rigid extension of $X$ in $\mathsf{OSp}^{G}_{(\bbf_1,\bbf_2)}$, and hence an injective envelope of $X$ in $\mathsf{OSp}^{G}_{(\bbf_1,\bbf_2)}.$
\end{proof}

\begin{remark}
    For the functors going from the unital operator space categories to the operator space categories, variants of Lemma~\ref{ffinj2} and Theorem~\ref{ugopstogopsp} also hold and follow directly from them, taking $G=\{e\}$. Alternatively, one can prove the results directly, considering the appropriate objects and morphisms. 
\end{remark}

%-----Subsection-----%

\subsection{From $G$-operator system categories to operator system categories} Assume $G$ is a finite group and consider the functor $\mcf_0$ from category $\mathsf{OSy}^{G}_{(\bbf_1,\bbf_2)}$ to the category $\mathsf{OSy}_{(\bbf_1,\bbf_2)}$. Only the intermediate case is treated below. For the real and complex $G$-categories, see \cite[Theorem~4.16]{bck23}.

\begin{thm}
    Let $X\in\mathsf{OSy}^{G}_{(\bbc,\bbr)}$ and let $(Z,\kappa)$ be an injective envelope of $X$. Then $(\mcf_0(Z),\mcf_0(\kappa))$ is an injective envelope of $\mcf_0(X)$. 
\end{thm}

%ADD REMARK ABOUT COMPOSITION OF FUNCTORS

\begin{proof}
    Let $(Z,\kappa)$ be an injective envelope of $X$ in $\mathsf{OSy}^{G}_{(\bbc,\bbr)}$. Then $Z$ is equivalently injective in $\mathsf{OSy}_{(\bbc,\bbr)}$ by Proposition~\ref{crginjiff}, so it suffices to show that $(Z,\kappa)$ is a rigid extension of $X$ in $\mathsf{OSy}_{(\bbc,\bbr)}$. Consider $X$ and $Z$ as operator spaces and consider $\kappa:X\to Z$ as a unital complete isometry. Assume $T:Z\to Z$ such that $T\circ\kappa=\kappa$. Apply the ``averaging" process to the maps; $\kappa$ stays the same since it was $G$-equivariant by assumption and $T$ becomes $\tilde{T}:Z\to Z$ such that $\tilde{T}\circ\kappa=\kappa$. Because $(Z,\kappa)$ is an injective envelope of $X$ in $\mathsf{OSy}^{G}_{(\bbc,\bbr)}$, it is a rigid extension of $X$ in $\mathsf{OSy}^{G}_{(\bbc,\bbr)}$, and so it follows that $\tilde{T}=\textup{Id}_{Z}$.
    By \cite[Corollary~2.1.42]{extreme}, the identity is an extreme point of the unit ball of any unital (real or complex) Banach algebra, so it follows that $T=\textup{Id}_Z$. Hence $(Z,\kappa)$ is a rigid extension of $X$ in $\mathsf{OSy}_{(\bbc,\bbr)}$ and thus is an injective envelope of $X$ in $\mathsf{OSy}_{(\bbc,\bbr)}$.
\end{proof}

%----New Section-----%

\section{Other functors between operator categories}\label{s4}

%-----Subsection-----%
\subsection{From intermediate categories to complex categories} Consider the complexification functor $\mcf_c$ from the category $\mathsf{OSp}_{(\bbc,\bbr)}$ (resp.\ $\mathsf{OSp}^{\1}_{(\bbc,\bbr)}$, $\mathsf{OSy}_{(\bbc,\bbr)}$) to the category $\mathsf{OSp}_{(\bbc,\bbc)}$ (resp.\ $\mathsf{OSp}^{\1}_{(\bbc,\bbc)}$, $\mathsf{OSy}_{(\bbc,\bbc)}$).

\begin{remark}\label{f00}
    The following lemma is not part of the ``functorial picture", as there is no functor in the diagram that is directly from the intermediate categories to the real categories. One could go from the intermediate categories to the real categories by composing functors, but there is something more natural; one can ``jump'' between categories. This jump is in fact a forgetful functor, say $\mcf_{00}$, from $\mathsf{OSp}_{(\bbc,\bbr)}$ to $\mathsf{OSp}_{(\bbr,\bbr)}$ (resp.\ from $\mathsf{OSp}^{\1}_{(\bbc,\bbr)}$ to $\mathsf{OSp}^{\1}_{(\bbr,\bbr)}$, from $\mathsf{OSy}_{(\bbc,\bbr)}$ to $\mathsf{OSy}_{(\bbr,\bbr)}$), which maps morphisms to themselves and maps objects to themselves, but considered as real spaces instead of complex spaces.
\end{remark}

\begin{lemma}\label{inttoreal}
    Let $X$ in $\mathsf{OSp}_{(\bbc,\bbr)}$ (resp.\ $\mathsf{OSp}^{\1}_{(\bbc,\bbr)}$, $\mathsf{OSy}_{(\bbc,\bbr)}$) be injective and consider the functor $\mcf_{00}$ from Remark~\ref{f00}. Then $\mcf_{00}(X)$ is injective.
\end{lemma}

\begin{proof}
    Note that $X\subseteq B(\mch)$ for some complex Hilbert space $\mch$ and $B(\mch)$ is a complex operator space which is complex injective. By \cite[Lemma~4.1]{bck23}, it is equivalently real injective. Since $X$ is intermediate injective, by Corollary~\ref{intermedinjiff}, there exists a completely contractive (resp.\ unital completely contractive, unital completely positive) idempotent $\bbr$-linear map $T$ from $B(\mch)$ onto $X$. Next, let $Y$ and $Z$ be real operator spaces (resp.\ unital operator spaces, operator systems) with a complete isometry (resp.\ unital complete isometry) $\kappa:Y\to Z$, and let $\vphi:Y\to X$ be a completely contractive (resp.\ unital completely contractive, unital completely positive) map. Then $\iota\circ\vphi:Y\to B(\mch)$ is a completely contractive (resp.\ unital completely contractive, unital completely positive) $\bbr$-linear map, where $\iota$ is the inclusion map. Since $B(\mch)$ is a real injective operator space (resp.\ unital operator space, operator system), there exists $\tilde{\vphi}:Z\to B(\mch)$ such that $\tilde{\vphi}\circ\kappa=\iota\circ\vphi$ and $\tilde{\vphi}$ is completely contractive (resp.\ unital completely contractive, unital completely positive). Since the image of $\vphi$ is contained in $X$, $T$ is the identity on the image of $\vphi$. Hence, $T\circ\tilde{\vphi}\circ\kappa=T\circ\iota\circ\vphi=\vphi$. Hence, $X$ is injective in $\mathsf{OSp}_{(\bbr,\bbr)}$ (resp.\ $\mathsf{OSp}^{\1}_{(\bbr,\bbr)}$, $\mathsf{OSy}_{(\bbr,\bbr)}$).
\end{proof}

\begin{cor}\label{f00injenv}
    Let $X\in\mathsf{OSp}_{(\bbc,\bbr)}$ (resp.\ in $\mathsf{OSp}^{\1}_{(\bbc,\bbr)}$, $\mathsf{OSy}_{(\bbc,\bbr)}$ and let $(Z,\kappa)$ be an injective envelope of $X$. Then $(\mcf_{00}(Z),\mcf_{00}(\kappa))$ is an injective envelope for $\mcf_{00}(X)$.
\end{cor}

\begin{proof}
    By Remark~\ref{f00}, $Z$ is injective in $\mathsf{OSp}_{(\bbr,\bbr)}$ (resp.\ $\mathsf{OSp}^{\1}_{(\bbr,\bbr)}$, $\mathsf{OSy}_{(\bbr,\bbr)}$). Since $\kappa$ is a completely isometric $\bbr$-linear map by assumption, it suffices to show that $(Z,\kappa)$ is a rigid extension of $X$ in $\mathsf{OSp}_{(\bbr,\bbr)}$ (resp.\ $\mathsf{OSp}^{\1}_{(\bbr,\bbr)}$, $\mathsf{OSy}_{(\bbr,\bbr)}$). Let $T:Z\to Z$ be a completely contractive (resp.\ unital completely contractive, unital completely positive) $\bbr$-linear map such that $T\circ\kappa=\kappa$. Since $(Z,\kappa)$ is a rigid extension of $X$ in $\mathsf{OSp}_{(\bbc,\bbr)}$ (resp.\ $\mathsf{OSp}^{\1}_{(\bbc,\bbr)}$, $\mathsf{OSy}_{(\bbc,\bbr)}$), it follows that $T=\textup{Id}_Z$. Hence $(Z,\kappa)$ is a rigid extension of $X$, and therefore an injective envelope of $X$, in $\mathsf{OSp}_{(\bbr,\bbr)}$ (resp.\ $\mathsf{OSp}^{\1}_{(\bbr,\bbr)}$, $\mathsf{OSy}_{(\bbr,\bbr)}$). 
\end{proof}

\begin{cor}\label{xccplxing}
    If $X$ in $\mathsf{OSp}_{(\bbc,\bbr)}$ (resp.\ $\mathsf{OSp}^{\1}_{(\bbc,\bbr)}$, $\mathsf{OSy}_{(\bbc,\bbr)}$) is injective, then $\mcf_c(X)$ is injective in $\mathsf{OSp}_{(\bbc,\bbc)}$ (resp.\ $\mathsf{OSp}^{\1}_{(\bbc,\bbc)}$, $\mathsf{OSy}_{(\bbc,\bbc)}$).
\end{cor}

\begin{proof}
    Follows by Lemma~\ref{inttoreal} and \cite[Proposition~4.1]{ruan03b}.
\end{proof}

\begin{remark}
    If $X$ in $\mathsf{OSp}_{(\bbc,\bbr)}$ (resp.\ $\mathsf{OSp}^{\1}_{(\bbc,\bbr)}$, $\mathsf{OSy}_{(\bbc,\bbr)}$) is injective, then $X$ is also injective in $\mathsf{OSp}_{(\bbc,\bbc)}$ (resp.\ $\mathsf{OSp}^{\1}_{(\bbc,\bbc)}$, $\mathsf{OSy}_{(\bbc,\bbc)}$) by Lemma~\ref{inttoreal} and \cite[Lemma~4.1]{bck23}.
\end{remark}

\begin{thm}\label{inttocplx}
    Let $X$ in $\mathsf{OSp_{(\bbc,\bbr)}}$ (resp.\ $\mathsf{OSp}^{\1}_{(\bbc,\bbr)}$, $\mathsf{OSy}_{(\bbc,\bbr)}$) and let $(Z,\kappa)$ be an injective envelope of $X$. Then $(\mcf_c(Z),\mcf_c(\kappa))$ is an injective envelope of $\mcf_c(X)$.
\end{thm}

\begin{proof}
    Let $X$ be a complex operator space (resp.\ unital operator space, operator system) and let $(Z,\kappa)$ be an injective envelope of $X$ in $\mathsf{OSp_{(\bbc,\bbr)}}$ (resp.\ $\mathsf{OSp}^{\1}_{(\bbc,\bbr)}$, $\mathsf{OSy}_{(\bbc,\bbr)}$). By Corollary~\ref{f00injenv}, $(Z,\kappa)$ is an injective envelope of $X$ in $\mathsf{OSp_{(\bbr,\bbr)}}$ (resp.\ $\mathsf{OSp}^{\1}_{(\bbr,\bbr)}$, $\mathsf{OSy}_{(\bbr,\bbr)}$). By Corollary~\ref{xccplxing}, $Z_c$ is injective in $\mathsf{OSp_{(\bbc,\bbc)}}$ (resp.\ $\mathsf{OSp}^{\1}_{(\bbc,\bbc)}$, $\mathsf{OSy}_{(\bbc,\bbc)}$). Further, $\kappa_c$ is a $\bbc$-linear complete isometry by \cite[Theorem~2.1]{ruan03b}. 
    Thus by \cite[Theorem~4.2]{bck23}, it follows that $(Z_c,\kappa_c)$ is an injective envelope for $X_c$. 
\end{proof}

\subsection{From intermediate $G$-categories to complex $G$-categories}

Consider the complexification functor $\mcf_c$ from the category $\mathsf{OSp}^G_{(\bbc,\bbr)}$ (resp.\ $\mathsf{OSp}^{\1,G}_{(\bbc,\bbr)}$, $\mathsf{OSy}^G_{(\bbc,\bbr)}$) to the category $\mathsf{OSp}^G_{(\bbc,\bbc)}$ (resp.\ $\mathsf{OSp}^{\1,G}_{(\bbc,\bbc)}$, $\mathsf{OSy}^G_{(\bbc,\bbc)}$). Previously in the paper, it was possible to prove results in the $G$-categories and the results in the non-$G$-categories followed from those, taking $G=\{e\}$. But, some of the following $G$-category facts rely on the facts of the corresponding non-$G$-category, so the proofs will be given.  

\begin{lemma}\label{ginttogreal}
    Let $X\in\mathsf{OSp}^G_{(\bbc,\bbr)}$ (resp.\ $\mathsf{OSp}^{\1,G}_{(\bbc,\bbr)}$, $\mathsf{OSy}^G_{(\bbc,\bbr)}$) be injective and let $\mcf_{00}$ be as in Remark~\ref{f00}, but now between the $G$-categories. Then $\mcf_{00}(X)$ is injective.
\end{lemma}

\begin{proof}
    If $X$ is injective in $\mathsf{OSp}^G_{(\bbc,\bbr)}$ (resp.\ $\mathsf{OSp}^{\1,G}_{(\bbc,\bbr)}$, $\mathsf{OSy}^G_{(\bbc,\bbr)}$), then $X$ is injective in $\mathsf{OSp}_{(\bbc,\bbr)}$ (resp.\ $\mathsf{OSp}^{\1}_{(\bbc,\bbr)}$, $\mathsf{OSy}_{(\bbc,\bbr)}$), and there exists a $G$-equivariant completely contractive (resp.\ unital completely contractive, unital completely positive) idempotent $\bbr$-linear map $T:\ell^{\infty}(G,X)\to X$. By Lemma~\ref{inttoreal}, $X$ is injective in $\mathsf{OSp}_{(\bbr,\bbr)}$ (resp.\ $\mathsf{OSp}^{\1}_{(\bbr,\bbr)}$, $\mathsf{OSy}_{(\bbr,\bbr)}$). Considering $X$ as a real space, $\ell^{\infty}(G,X)$ is a real space and there exists a $G$-equivariant completely contractive (resp.\ unital completely contractive, unital completely positive) idempotent $\bbr$-linear map $T:\ell^{\infty}(G,X)\to X$ by above. But this is equivalent to $X$ being injective in $\mathsf{OSp}^G_{(\bbr,\bbr)}$ (resp.\ $\mathsf{OSp}^{\1,G}_{(\bbr,\bbr)}$, $\mathsf{OSy}^G_{(\bbr,\bbr)}$) by \cite[Lemma~4.6]{bck23}.
\end{proof}

\begin{cor}\label{f00ginjenv}
    Let $X\in\mathsf{OSp}^G_{(\bbc,\bbr)}$ (resp.\ in $\mathsf{OSp}^{\1,G}_{(\bbc,\bbr)}$, $\mathsf{OSy}^G_{(\bbc,\bbr)}$ and let $(Z,\kappa)$ be an  injective envelope of $X$. Then $(\mcf_{00}(Z),\mcf_{00}(\kappa))$ is an injective envelope for $\mcf_{00}(X)$.
\end{cor}

\begin{proof}
    Let $X$ in $\mathsf{OSp}^G_{(\bbc,\bbr)}$ (resp.\ $\mathsf{OSp}^{\1,G}_{(\bbc,\bbr)}$, $\mathsf{OSy}^G_{(\bbc,\bbr)}$) and let $(Z,\kappa)$ be an injective envelope of $X$. Since $\kappa$ is a $G$-equivariant completely isometric $\bbr$-linear map by assumption, it suffices to show that $(Z,\kappa)$ is a rigid extension in $\mathsf{OSp}^G_{(\bbr,\bbr)}$ (resp.\ $\mathsf{OSp}^{\1,G}_{(\bbr,\bbr)}$, $\mathsf{OSy}^G_{(\bbr,\bbr)}$). It follows similarly as in Corollary~\ref{f00injenv} that $(Z,\kappa)$ is a rigid extension, and hence an injective envelope, of $X$ in $\mathsf{OSp}^G_{(\bbr,\bbr)}$ (resp.\ $\mathsf{OSp}^{\1,G}_{(\bbr,\bbr)}$, $\mathsf{OSy}^G_{(\bbr,\bbr)}$).
\end{proof}

\begin{cor}
    If $X\in\mathsf{OSp}^G_{(\bbc,\bbr)}$ (resp.\ $\mathsf{OSp}^{\1,G}_{(\bbc,\bbr)}$, $\mathsf{OSy}^G_{(\bbc,\bbr)}$) is injective, then $\mcf_c(X)$ is injective.
\end{cor}

\begin{proof}
    Follows by the Lemma~\ref{ginttogreal} and by \cite[Corollary~4.8]{bck23}.
\end{proof}

\begin{remark}
    If $X\in\mathsf{OSp}^G_{(\bbc,\bbr)}$ (resp.\ $\mathsf{OSp}^{\1,G}_{(\bbc,\bbr)}$, $\mathsf{OSy}^G_{(\bbc,\bbr)}$) is injective, then $X$ is injective in $\mathsf{OSp}^G_{(\bbc,\bbc)}$ (resp.\ $\mathsf{OSp}^{\1,G}_{(\bbc,\bbc)}$, $\mathsf{OSy}^G_{(\bbc,\bbc)}$) by Lemma~\ref{ginttogreal} and \cite[Corollary~4.12]{bck23}.
\end{remark}

\begin{thm}
    Let $X\in\mathsf{OSp}^G_{(\bbc,\bbr)}$ (resp.\ $\mathsf{OSp}^{\1,G}_{(\bbc,\bbr)}$, $\mathsf{OSy}^G_{(\bbc,\bbr)}$) and let $(Z,\kappa)$ be an injective envelope of $X$. Then $(\mcf_c(Z),\mcf_c(\kappa
    ))$ is an injective envelope of $\mcf_c(X)$.
\end{thm}

\begin{proof}
    Follows as in Theorem~\ref{inttocplx}, considering the appropriate objects and morphisms and using Corollary~\ref{f00ginjenv}. 
\end{proof}

%-----Subsection-----%

%Fill in details
\subsection{From unital operator space categories to operator system categories} Consider the Arveson functor $\mcf_A$ from the category $\mathsf{OSp}^{\1}_{(\bbf_1,\bbf_2)}$ to the category $\mathsf{OSy}_{(\bbf_1,\bbf_2)}$. Here only the intermediate case is considered. For the real and complex cases, see \cite[Proposition~4.9]{bck23} taking $G=\{e\}$.

The following theorem is the intermediate version of the well-known result of Choi-Effros in the complex case (\cite{choieffros}).

\begin{thm}\label{enviscalg1}
	Let $X\in\mathsf{OSp}^{\1}_{(\bbc,\bbr)}$ be injective and let $\phi:B(\mch)\to X$ be a unital completely contractive idempotent $\bbr$-linear map onto $X$ for some complex Hilbert space $\mch$. Then setting $x\circ y\coloneqq\phi(xy)$ defines a multiplication on $X$, and $X$ together with this multiplication and its usual $\ast$-operation is a complex $C^*$-algebra. 
\end{thm}

\begin{proof}
    Note that $X\subseteq B(\mch)$ for some complex Hilbert space $\mch$ and a unital completely contractive idempotent $\bbr$-linear map from $B(\mch)$ onto $X$ always exists by using injectivity to extend $\textup{Id}_X$. So, consider $X$ and $B(\mch)$ to be real unital operator spaces. Then both are injective in $\mathsf{OSp}^{\1}_{(\bbr,\bbr)}$ by Lemma~\ref{inttoreal}. By \cite[Theorem~4.9]{sharma14}, $X$ (resp.\ $B(\mch)$) is a unital real $C^*$-algebra. By \cite[Theorem 4.14]{bck23}, $X$ (resp.\ $B(\mch)$ is a unital complex $C^*$-algebra.
\end{proof}

\begin{prop}\label{mcfainj}
    If $X$ is in $\mathsf{OSp}^{\1}_{(\bbc,\bbr)}$ is injective, then $\mcf_A(X)$ is injective.
\end{prop}

\begin{proof}
    Since $X$ is injective, by Theorem~\ref{enviscalg1}, it is a (complex) $C^*$-algebra. This means that $X+X^{\star}=X$. Let $Y$ and $W$ be unital complex operator space, let $\iota:Y\to W$ be a unital complete isometry, and let $\vphi:Y\to X+X^{\star}$ be a unital complete isometry. Then, noting that every unital completely positive map is a unital completely contractive map and considering all the spaces as unital operator spaces, there exists a unital completely contractive extension $\tilde{\vphi}:W\to X=X+X^{\star}$. But $\tilde{\vphi}$ is equivalently a unital completely positive map. It follows that $X+X^{\star}$ is an injective operator system.
\end{proof}

\begin{thm}\label{ginjenvelope}
	If $X$ is in $\mathsf{OSp}^{\1}_{(\bbc,\bbr)}$ and $(Z,\kappa)$ is an injective envelope of $X$, then $(\mcf_A(Z),\mcf_A(\kappa))$ is an injective envelope of $\mcf_A(X)$.
\end{thm}

\begin{proof}
    By Proposition~\ref{mcfainj} and \cite[Lemma 1.3.6]{blm04}, $(Z+Z^{\star},\tilde{\kappa})=(Z,\tilde{\kappa})$ is an injective extension of $X+X^{\star}$. So it suffices to show that it is a rigid extension of $X+X^{\star}$. Let $T:Z\to Z$ be a unital completely positive map such that $T\circ\tilde{\kappa}=\tilde{\kappa}$. Note that $(T\circ\tilde{\kappa})\arrowvert_{X}=T\circ\kappa\arrowvert_X=T\circ\kappa$ and that $\tilde{\kappa}\arrowvert_X=\kappa$ and therefore $T\circ\kappa=\kappa$. Since $(Z,\kappa)$ is a rigid extension of $X$, $T=\textup{Id}_Z$, which implies that $(Z,\tilde{\kappa})$ is a rigid extension of $X+X^{\star}$. Thus $(Z,\tilde{\kappa})$ is an injective envelope of $X+X^{\star}$.
\end{proof}

\begin{remark}\label{r1}
    If $X$ is a unital operator space, then any injective envelope $(W,j)$ of $X$ in the category of unital operator spaces has a unique product and unique involution with respect to which it is a unital $C^*$-algebra (the uniqueness may be seen by \cite[Corollary 1.3.10]{blm04}).
\end{remark}

\begin{lemma}\label{ieopsyforuopsp}
    Let $(Z,\kappa)$ be the injective envelope of $\mcf_A(X)$. Then $(Z,\kappa\arrowvert_X)$ is an injective envelope of $X$.
\end{lemma}

\begin{proof}
    Let $(Z,\kappa)$ be an injective envelope of $X+X^{\star}$ in $\mathsf{OSy}_{(\bbc,\bbr)}$. Then $Z$ is an injective unital operator space by Lemma~\ref{ffinj} (letting $G=\{e\}$) and $\kappa\arrowvert_X$ is a unital complete isometry. Therefore, $(Z,\kappa\arrowvert_X)$ is an injective extension, so it suffices to show that it is a rigid extension of $X$. To see this, let $T:Z\to Z$ be a unital completely contractive map such that $T\circ\kappa\arrowvert_X=\kappa\arrowvert_X$. Note that $(T\circ\kappa)\arrowvert_X=T\circ\kappa\arrowvert_X$, so $\widetilde{(T\circ\kappa)\arrowvert_X}=\widetilde{T\circ\kappa\arrowvert_X}$. But $T$ and $\kappa$ are both equivalently unital completely positive maps, and hence $^*$-linear. Therefore $\widetilde{(T\circ\kappa)\arrowvert_X}=T\circ\kappa$ and $\widetilde{\kappa\arrowvert_X}=\kappa$. It follows that $T\circ\kappa=\kappa$ and since $(Z,\kappa)$ is a rigid extension of $X+X^{\star}$, it follows that $T=\textup{Id}_Z$. But this implies that $(Z,\kappa\arrowvert_X)$ is a rigid extension of $X$. Thus $(Z,\kappa\arrowvert_X)$ is an injective envelope of $X$.
\end{proof}

\begin{remark}
    Considering Theorem~\ref{ginjenvelope}, Remark~\ref{r1}, and Lemma~\ref{ieopsyforuopsp} together, any injective envelope of an object $X$ in either a unital operator space category or an operator system category can be considered as either a $C^*$-algebra or an operator system and will be treated as such without comment.
\end{remark}

%-----Subsection-----%

\subsection{From unital $G$-operator space categories to $G$-operator system categories} Consider the Arveson functor $\mcf_A$ from category $\mathsf{OSp}^{\1,G}_{(\bbf_1,\bbf_2)}$ to $\mathsf{OSy}^{G}_{(\bbf_1,\bbf_2)}$. Here only the intermediate case is considered. For the real and complex cases, see \cite[Proposition~4.9]{bck23}.

\begin{thm}
	Let $X\in\mathsf{OSp}^{\1,G}_{(\bbc,\bbr)}$ be injective. Let $\phi:\ell^{\infty}(G,X)\to X$ be a $G$-equivariant unital completely contractive idempotent $\bbr$-linear map onto $X$. Then setting $x\circ y\coloneqq\phi(xy)$ defines a multiplication on $X$, and $X$ together with this multiplication and it is usual $\ast$-operation is a complex $G$-$C^*$-algebra.
\end{thm}

\begin{proof}\label{enviscalg2}
    Note that the map $\phi$ always exists by the injectivity of $X$ in $\mathsf{OSp}^{\1,G}_{(\bbc,\bbr)}$. By Theorem~\ref{enviscalg2}, $X$ (resp.\ $\ell^{\infty}(G,X)$) is a unital complex $C^*$-algebra. It suffices to show that it is a unital complex $G$-$C^*$-algebra. Firstly, for $g\in G$ and $x,y\in X$, $g(x\circ y)=g\phi(xy)=\phi(g(xy))=\phi((gx)(gy))=(gx)\circ (gy).$ Secondly, for $x,y\in X$, $(x\circ y)^*=\phi(xy)^*=\phi((xy)^*)=\phi(y^*x^*)=y^*\circ x^*$.  
\end{proof}

\begin{prop}\label{mcfaginj}
    If $X\in\mathsf{OSp}^{\1,G}_{(\bbc,\bbr)}$ is injective, then $\mcf_A(X)$ is injective.
\end{prop}

\begin{proof}
    Since $X$ is injective in $\mathsf{OSp}^{\1,G}_{(\bbc,\bbr)}$, by Theorem~\ref{enviscalg2}, $X$ is a $C^*$-algebra. This means that $X+X^{\star}=X$. Since every $G$-equivariant unital completely positive map is a $G$-equivariant unital completely contractive map, it follows that $X+X^{\star}$ is injective in $\mathsf{OSy}^{G}_{(\bbc,\bbr)}$.
\end{proof}

\begin{thm}
	Let $X\in\mathsf{OSp}^{\1,G}_{(\bbc,\bbr)}$ and let $(Z,\kappa)$ be an injective envelope of $X$. Then $(\mcf_A(Z),\mcf_A(\kappa))$ is an injective envelope of $\mcf_A(X)$.
\end{thm}

\begin{proof}
    Follows similarly to the proof of Theorem~\ref{ginjenvelope}, using Proposition~\ref{mcfaginj} and \cite[Lemma 1.3.6]{blm04}.
\end{proof}

\begin{lemma}
    Let $X\in\mathsf{OSp}^{\1,G}_{(\bbc,\bbr)}$ and let $(Z,\kappa)$ be an injective envelope of $\mcf_A(X)$. Then $(Z,\kappa\arrowvert_X)$ is an injective envelope of $X$. 
\end{lemma}

\begin{proof}
    Follows as in Lemma~\ref{ieopsyforuopsp}, considering the appropriate objects and morphisms.
\end{proof}

%-----Subsection-----%

\subsection{From operator system categories to $G$-operator system categories} Let $G$ be a finite group and consider the functor $\mcf_G$ from the category $\mathsf{OSy}_{(\bbf_1,\bbf_2)}$ to the category $\mathsf{OSy}^{G}_{(\bbf_1,\bbf_2)}$.

\begin{thm}\label{trivginnj}
	Let $X\in\mathsf{OSy}_{(\bbf,\bbf)}$ and let $(Z,\kappa)$ be an injective  envelope of $X$. Then $(\mcf_G(Z),\mcf_G(\kappa))$ is an injective envelope of $\mcf_G(X)$.
\end{thm}

\begin{proof}
	Assume $(Z,\kappa)$ is an injective envelope of $X$ in $\mathsf{OSy}_{(\bbf,\bbf)}$ and let $G$ act trivially on $X$ (and therefore also on $Z$). Then for any two $G$-operator systems $Y$ and $W$ such that $\phi:Y\to Z$ is a $G$-equivariant unital completely positive $\bbf$-linear map and $\iota:Y\to W$ is a $G$-equivariant unital $\bbf$-linear complete isometry, there exists a unital completely positive $\bbf$-linear map $\hat{\phi}:W\to Z$, when considering all spaces as operator systems and ignoring the $G$-equivariance of the maps. Next, consider the new diagram:
		\begin{center}
			\begin{tikzcd}
				Y \arrow[r, "\tilde{\phi}"] \arrow[d, "\tilde{\iota}",labels=left] & Z\\
				W \arrow[from=2-1,to=1-2,"\enspace\tilde{\hat{\phi}}",labels=below]& 
			\end{tikzcd}
		\end{center}
	where all maps are put through the ``averaging" process mentioned in \cite[Remark~3 after Lemma~4.4]{bck23}. But since $\phi$ and $\iota$ were already $G$-equivariant, $\tilde{\phi}=\phi$ and $\hat{\psi}=\psi$. Hence, we get that $Z$ is $G$-injective, or is equivalently injective in $\mathsf{OSy}^{G}_{(\bbf,\bbf)}.$ To show $(W,\kappa)$ is an injective envelope of $X$ in $\mathsf{OSy}^{G}_{(\bbc,\bbc)}$, it suffices to show that $(W,\kappa)$ is a rigid extension of $X$ in $\mathsf{OSy}^{G}_{(\bbf,\bbf)}$. Assume $T:Z\to Z$ is a $G$-equivariant unital completely positive $\bbf$-linear map such that $T\circ\kappa=\kappa$. By ignoring the $G$-equivariance of the maps and $G$-actions on the spaces, by using the injectivity of $Z$ in $\mathsf{OSy}_{(\bbf,\bbf)}$, it follows that $T=\textup{Id}_{Z}$. But since the maps are $G$-equivariant and the spaces are $G$-spaces, it follows that $(Z,\kappa)$ is a rigid extension of $X$ in $\mathsf{OSy}^{G}_{(\bbf,\bbf)}$, and hence is an injective envelope of $X$, in $\mathsf{OSy}^{G}_{(\bbf,\bbf)}$. 
\end{proof}

For the intermediate category, the same result holds using the previous theorem. 

\begin{thm}
	Let $X\in\mathsf{OSy}_{(\bbc,\bbr)}$ and let $(Z,\kappa)$ be an injective envelope of $X$. Then $(\mcf_G(Z),\mcf_G(\kappa))$ is an injective envelope of $\mcf_G(X)$.
\end{thm}

\begin{proof}
	Assume $(Z,\kappa)$ is an injective envelope of $X$ in $\mathsf{OSy}_{(\bbc,\bbr)}$. Then for any two complex $G$-operator systems $Y$ and $W$ such that $\phi:Y\to Z$ is a $G$-equivariant unital completely positive $\bbr$-linear map and $\iota:Y\to W$ is a $G$-equivariant unital $\bbr$-linear complete isometry, there exists a unital completely positive $\bbr$-linear map $\hat{\phi}:W\to Z$, when considering all spaces as real operator systems and ignoring the $G$-equivariance of the maps. By a similar argument for the real case, it follows that $(Z,\kappa)$ is a rigid extension of $X$, and hence an injective envelope of $X$, in $\mathsf{OSy}^G_{(\bbc,\bbr)}$.
\end{proof}

%----New Section-----%
\section{Properties of the considered functors}\label{s5}

Let $\mathsf{C},\mathsf{D}$ be two categories. Then a functor $F:\mathsf{C}\to\mathsf{D}$ is
\begin{itemize}
    \item \textit{faithful} when, for every pair of objects $X$ and $Y$ of $\mathsf{C}$ and every pair of morphisms $f,f':X\to Y\in\mathsf{C}$, $f=f'$ in $\mathsf{C}$ whenever $F(f)=F(f')$ in $\mathsf{D}$.
    
    \item \textit{full} when, for every pair of objects $X$ and $Y$ of $\mathsf{C}$ and every morphism $g:F(X)\to F(Y)\in\mathsf{D}$, there exists a morphism $f:X\to Y\in\mathsf{C}$ such that $F(f)=g$.

    \item \textit{essentially surjective} if for every object $Z$ of $\mathsf{D}$, there exists an object $X$ of $\mathsf{C}$ and an isomorphism in $\mathsf{D}$ from $F(X)$ to $Z$.
\end{itemize}

%-----Subsection-----%
\subsection{Faithfulness of the functors}

%-----Subsubsection-----%

\begin{remark}
    Every forgetful functor $\mcf_0$ is trivially faithful. All three $\mcf_G$ functors are trivially faithful.
\end{remark}

%-----Subsubsection-----%

\begin{thm}
    All six cases of the $\mcf_A$ functor are faithful.
\end{thm}

\begin{proof}
    Let $X$ and $Y$ be objects from $\mathsf{OSp}^{\1}_{(\bbr,\bbr)}$ (resp.\ from $\mathsf{OSp}^{\1}_{(\bbc,\bbc)}$, $\mathsf{OSp}^{\1}_{(\bbc,\bbr)}$) and $f,f':X\to Y$ be morphisms from $\mathsf{OSp}^{1}_{(\bbr,\bbr)}$ (resp.\ from $\mathsf{OSp}^{\1}_{(\bbc,\bbc)}$, $\mathsf{OSp}^{\1}_{(\bbc,\bbr)})$. Assume $\mcf_A(f:X\to Y)=\mcf_A(f':X\to Y)$. But $\mcf_A(f:X\to Y)=\mcf_A(f):X+X^{\star}\to Y+Y^{\star}$ and $\mcf_A(f':X\to Y)=\mcf_A(f'):X+X^{\star}\to Y+Y^{\star}$. Since $\mcf_A(f)=\mcf_A(f')$ and $X$ is a subset of $X+X^{\star}$, $\mcf_A(f)\arrowvert_X=\mcf_A(f')\arrowvert_X$. But $\mcf_A(f)\arrowvert_X=f$ and $\mcf_A(f')\arrowvert_X=f'$. Hence $f=f'$, and it follows that $\mcf_A$ is faithful. The proof for the $G$-categories follow similarly.
\end{proof}

\begin{thm}
    The functor $\mcf_c$ in all twelve cases is faithful.
\end{thm}

\begin{proof}
    Let $X$ and $Y$ be objects from $\mathsf{OSp}_{(\bbr,\bbr)}$ (resp.\ from $\mathsf{OSp}^{\1}_{(\bbr,\bbr)}$, $\mathsf{OSy}_{(\bbr,\bbr)}$) and let $f,f':X\to Y$ be morphisms from $\mathsf{OSp}_{(\bbr,\bbr)}$ (resp.\ from $\mathsf{OSp}^{\1}_{(\bbr,\bbr)}$, $\mathsf{OSy}_{(\bbr,\bbr)}$). Assume $\mcf_c(f:X\to Y)=\mcf_c(f':X\to Y)$. But $\mcf_c(f:X\to Y)=\mcf_c(f):X_c\to Y_c$ and $\mcf_c(f':X\to Y)=\mcf_c(f'):X_c\to Y_c$. Since $\mcf_c(f)=\mcf_c(f')$ and $X$ is a subset of $X_c$, $\mcf_c(f)\arrowvert_X=\mcf_c(f')\arrowvert_X$. But $\mcf_c(f)\arrowvert_X=f$ and $\mcf_c(f')\arrowvert_X=f'$. Hence $f=f'$, and it follows that $\mcf_c$ is faithful. The proof in the $G$-categories follow similarly.
\end{proof}

%-----Subsection-----%
\subsection{Fullness of the functors} 

\begin{thm}
    The functor $\mcf_0$ from $\mathsf{OSy}_{(\bbc,\bbc)}$ to $\mathsf{OSp}^{\1}_{(\bbc,\bbc)}$ (resp.\ from $\mathsf{OSy}_{(\bbr,\bbr)}$ to $\mathsf{OSp}^{\1}_{(\bbr,\bbr)}$, from $\mathsf{OSy}_{(\bbc,\bbr)}$ to $\mathsf{OSp}^{\1}_{(\bbc,\bbr)}$, from $\mathsf{OSy}^G_{(\bbc,\bbc)}$ to $\mathsf{OSp}^{\1,G}_{(\bbc,\bbc)}$, from $\mathsf{OSy}^G_{(\bbr,\bbr)}$ to $\mathsf{OSp}^{\1,G}_{(\bbr,\bbr)}$, from $\mathsf{OSy}^G_{(\bbc,\bbr)}$ to $\mathsf{OSp}^{\1,G}_{(\bbc,\bbr)}$) is full.
\end{thm}

\begin{proof}
    This is because every unital map is completely contractive if and only if it is completely positive.
\end{proof}

\begin{remark}
    The functor $\mcf_0$ in all other cases than those above are trivially not full.
\end{remark}

\begin{remark}
    The functor $\mcf_G$ in all three cases are trivially full.
\end{remark}

\begin{thm} The functor $\mcf_A$ from $\mathsf{OSp}^{1}_{(\bbr,\bbr)}$ to $\mathsf{OSy}_{(\bbr,\bbr)}$ (resp.\ from $\mathsf{OSp}^{\1}_{(\bbc,\bbc)}$ to $\mathsf{OSy}_{(\bbc,\bbc)}$, from $\mathsf{OSp}^{\1}_{(\bbc,\bbr)}$ to $\mathsf{OSy}_{(\bbc,\bbr)}$, from $\mathsf{OSp}^{1,G}_{(\bbr,\bbr)}$ to $\mathsf{OSy}^G_{(\bbr,\bbr)}$, from $\mathsf{OSp}^{\1,G}_{(\bbc,\bbc)}$ to $\mathsf{OSy}^G_{(\bbc,\bbc)}$, from $\mathsf{OSp}^{\1,G}_{(\bbc,\bbr)}$ to $\mathsf{OSy}^G_{(\bbc,\bbr)}$) is not full.
\end{thm}

\begin{proof}
    Consider the unital operator space of upper triangular matrices $$\mct_2=\{ \begin{bmatrix}
        a_{11} & a_{12}\\
        0 & a_{22}
    \end{bmatrix}: a_{11},a_{12},a_{22}\in\bbc\}$$ in $M_2(\bbc)$ and consider the map $g:M_2(\bbc)\to M_2(\bbc)$ defined by $$g(\begin{bmatrix}
        a_{11} & a_{12}\\
        a_{21} & a_{22}
    \end{bmatrix})=
    \begin{bmatrix}
        0 & 1\\
        1 & 0
    \end{bmatrix}
    \begin{bmatrix}
        a_{11} & a_{12}\\
        a_{21} & a_{22}
    \end{bmatrix}
    \begin{bmatrix}
        0 & 1\\
        1 & 0
    \end{bmatrix}=
    \begin{bmatrix}
        a_{22} & a_{21}\\
        a_{12} & a_{11}
    \end{bmatrix}
    $$ for all $\begin{bmatrix}
        a_{11} & a_{12}\\
        a_{21} & a_{22}
    \end{bmatrix}\in M_2(\bbc)$. This map is unital and a *-isomorphism, hence also completely positive. Note that $M_2(\bbc)=\mct_2+\mct_2^{\star}$. If this functor was full, then there would exist a unital completely contractive map $f:\mct_2\to\mct_2$ such that $\mcf_A(f)=g$. But $g(\mct_2)=f(\mct_2)\subseteq\mct_2$ and $g(\begin{bmatrix}
        0 & 1\\
        0 & 0
    \end{bmatrix})=\begin{bmatrix}
        0 & 0\\
        1 & 0
    \end{bmatrix}\not\in\mct_2$, a contradiction. Hence, the functor is not full. 
\end{proof} 

\begin{thm}
    None of the 12 complexification functors are full.
\end{thm}

\begin{proof}
    Follows by \cite[Corollary 2.4]{bck23} and in the $G$-case, \cite[Remark 2]{bck23}.
\end{proof}

%-----Subsection-----%
\subsection{Essential surjectiveness of the functors}

\begin{thm}
    The functor $\mcf_0$ from $\mathsf{OSp}_{(\bbc,\bbc)}$ to $\mathsf{OSp}_{(\bbr,\bbr)}$ (resp.\ from $\mathsf{OSp}^{\1}_{(\bbc,\bbc)}$ to $\mathsf{OSp}^{\1}_{(\bbr,\bbr)}$, from $\mathsf{OSy}_{(\bbc,\bbc)}$ to $\mathsf{OSy}_{(\bbr,\bbr)}$, from $\mathsf{OSp}^G_{(\bbc,\bbc)}$ to $\mathsf{OSp}^G_{(\bbr,\bbr)}$, from $\mathsf{OSp}^{\1,G}_{(\bbc,\bbc)}$ to $\mathsf{OSp}^{\1,G}_{(\bbr,\bbr)}$, from $\mathsf{OSy}^G_{(\bbc,\bbc)}$ to $\mathsf{OSy}^G_{(\bbr,\bbr)}$) is not essentially surjective.
\end{thm}

\begin{proof}
    There are real $C^*$-algebras which are not also complex $C^*$-algebras, giving examples for all the instances above; for example, the quaternions. Hence, the functor in the above cases are not essentially surjective. See \cite[Section~3]{b23} for more details.
\end{proof}

\begin{remark}
    The functor $\mcf_0$ from $\mathsf{OSp}_{(\bbc,\bbc)}$ to $\mathsf{OSp}_{(\bbc,\bbr)}$ (resp.\ from $\mathsf{OSp}^{\1}_{(\bbc,\bbc)}$ to $\mathsf{OSp}^{\1}_{(\bbc,\bbr)}$, from $\mathsf{OSy}_{(\bbc,\bbc)}$ to $\mathsf{OSy}_{(\bbc,\bbr)}$, from $\mathsf{OSp}^G_{(\bbc,\bbc)}$ to $\mathsf{OSp}^G_{(\bbc,\bbr)}$, from $\mathsf{OSp}^{\1,G}_{(\bbc,\bbc)}$ to $\mathsf{OSp}^{\1,G}_{(\bbc,\bbr)}$, from $\mathsf{OSy}^G_{(\bbc,\bbc)}$ to $\mathsf{OSy}^G_{(\bbc,\bbr)}$) is trivially essentially surjective.
\end{remark}

\begin{thm}
    The functor $\mcf_0$ from $\mathsf{OSp}^{\1}_{(\bbc,\bbc)}$ to $\mathsf{OSp}_{(\bbc,\bbc)}$ (resp.\ from $\mathsf{OSp}^{\1}_{(\bbc,\bbr)}$ to $\mathsf{OSp}_{(\bbc,\bbr)}$, from $\mathsf{OSp}^{\1}_{(\bbr,\bbr)}$ to $\mathsf{OSp}_{(\bbr,\bbr)}$, from $\mathsf{OSp}^{\1,G}_{(\bbc,\bbc)}$ to $\mathsf{OSp}^G_{(\bbc,\bbc)}$, from $\mathsf{OSp}^{\1,G}_{(\bbc,\bbr)}$ to $\mathsf{OSp}^G_{(\bbc,\bbr)}$, from $\mathsf{OSp}^{\1,G}_{(\bbr,\bbr)}$ to $\mathsf{OSp}^G_{(\bbr,\bbr)}$) is not essentially surjective.
\end{thm}

\begin{proof}
    Consider the column space $$\mcc=\{ \begin{bmatrix}
    a_{11} & 0 \\
    a_{21} & 0
    \end{bmatrix} : a_{11},a_{22}\in\bbc
    \}$$ of $M_2(\bbc)$. If $\mcc\cong Z$ as operator spaces for some unital operator space $Z$, then their ternary envelopes are isomorphic. But $\mct(\mcc)=\mcc$ and $\mct(Z)$ is a $C^*$-algebra. Note that $\mcc$ is not a $C^*$-algebra since it is not self-adjoint. Furthermore, it is not completely isometric to a $C^*$-algebra because it is 2-dimensional and a Hilbert space, while the only 2-dimensional $C^*$-algebra is $\ell^{\infty}_2$ and it is not a Hilbert space.
\end{proof}

\begin{remark}
    The functor $\mcf_0$ from $\mathsf{OSy}^G_{(\bbc,\bbc)}$ to $\mathsf{OSy}_{(\bbc,\bbc)}$ (resp.\ from $\mathsf{OSy}^G_{(\bbc,\bbr)}$ to $\mathsf{OSy}_{(\bbc,\bbr)}$, from $\mathsf{OSy}^G_{(\bbr,\bbr)}$ to $\mathsf{OSy}_{(\bbr,\bbr)}$) is trivially essentially surjective.
\end{remark}

\begin{thm}
    All three cases of the functor $\mcf_G$ are not essentially surjective.
\end{thm}

\begin{proof}
    Let $Z$ be a real or complex $G$-operator system with non-trivial $G$-action. For any real or complex operator system $X$, applying the functor to $X$ makes it a $G$-operator system with the trivial  $G$-action, so $Z\not\cong\mcf_G(X)$.
\end{proof}

%-----Subsubsection-----%
%\subsubsection{Arveson functor} 

\begin{thm}
    The functor $\mcf_A$ is essentially surjective in all six cases.
\end{thm}

\begin{proof}
    Let $Z$ be a real or complex operator system and consider it as a unital operator space $X=\mcf_0(Z)$. Then $\mcf_A(X)\cong Z$.
\end{proof}

\begin{thm}
    All twelve cases of the functor $\mcf_c$ are not essentially surjective.
\end{thm}

\begin{proof}
    For the functor going from the real categories to the complex categories, see \cite[Remark 2 of Section 2]{bck23}. The cases of the functor going from the intermediate categories to the complex categories follows by \cite[Lemma 5.1]{b23}, which says $\mcf_c(X)=X\oplus_{\infty}\bar{X}$ for a complex operator space (resp.\ unital operator space, operator system). But not every complex operator space (resp.\ unital operator space, operator system) is of this form. For example, there is no complex operator space (resp.\ unital operator space, operator system)$Z$ such that $\bbc\cong Z\oplus_{\infty} \bar{Z}$. Hence the functor in all of these cases are not essentially surjective.
\end{proof}

%-----Section-----%

\section{Regarding the definition of an injective envelope}\label{s6}

Many of the notions considered throughout this paper are purely categorical and can be explored for any category one would like to examine. That is, an object $Z$ in a category $\mathsf{C}$ is \textit{injective} if for any objects $X$ and $Y$ in $\mathsf{C}$, any morphism $f:X\to Z$, and any monomorphism $\iota :X\to Y$, there exists a morphism $\tilde{f}:Y\to Z$ such that $\tilde{f}\circ\iota=f$. A pair $(Z,\iota)$ is an \textit{injective envelope} of an object $X\in\mathsf{C}$ if $Z$ is an injective object in $\mathsf{C}$ and $\iota:X\to Z\in\mathsf{C}$ is an essential embedding. The terms ``monomorphism'' and ``essential'' are also purely categorical.

For objects $X,Y,$ and $Z$ in a category $\mathsf{C}$, a morphism $f:X\to Y$ in $\mathsf{C}$ is:
\begin{itemize}
    \item a \textit{monomorphism} if for every object $Z\in\mathsf{C}$ and every pair of morphisms $g_1,g_2:Z\to X\in\mathsf{C}$, then $(f\circ g_1=f\circ g_2)\implies (g_1=g_2)$
    
    \item \textit{essential} if any morphism $g:Y\to Z\in\mathsf{C}$ is a monomorphism if and only if $g\circ f$ is a monomorphism.
\end{itemize} 

In the category $\mcb$ consisting of Banach spaces as objects and contractive continuous linear maps as morphisms, Pothoven showed in \cite[Proposition~1.13]{PothovenThesis} that morphisms in this category are monomorphisms if and only if they are one-to-one. Furthermore, he showed in \cite{Pothoven69} that extra conditions must be imposed on the monomorphisms in the definition of an injective envelope, or the only injective object is $\{0\}$. 

Throughout this paper and in Hamana's settings, there are extra conditions are imposed on the monomorphisms. It follows similarly to Pothoven's results that if these extra conditions are not imposed on the monomorphisms in these categories, then there will not be enough (or any) injective objects. The proofs of the following results are modeled after Pothoven's techniques in \cite{Pothoven69}. 

Note that a subspace $A$ of a topological space $X$ is called a \textit{retract of $X$} if there is a continuous map $f:X\to X$ (called a \textit{retraction}) such that:
\begin{enumerate}
    \item $f(x) \in A$ for all $x \in X$ and
    \item $f(a)=a$ all $a \in A$.
\end{enumerate}

\begin{lemma}
	$\bbc$ is a retract of every complex operator space. 
\end{lemma}

\begin{proof}
	Let $X$ be a nonzero operator space. Fix $x_0\in X\setminus\{0\}$ and let $x_0^*:X\to\bbc$ be a bounded linear map with $\norm{x_0^*}=1$ and $x_0^*(x_0)=\norm{x_0}$. Then $x_0^*$ is completely contractive and $x_0^*(x_0/\norm{x_0})=1.$ Define $f:\bbc\to X$ by $f(m)=mx_0/\norm{x_0}$. Then $f$ is completely contractive since it is contractive. Lastly, $x_0^*\circ f=\textup{Id}_{\bbc}$, hence $\bbc$ is a retract of $X$.
\end{proof}

\begin{lemma}
	In the category of operator spaces and completely contractive maps, $\bbc$ is not injective (in the purely categorical sense).
\end{lemma} 

\begin{proof}
	Consider $\ell^1_{\bbn}(\bbc)$ with its dual operator space structure and define the map $g:Max(\ell^1_{\bbn}(\bbc))\to\ell^1_{\bbn}(\bbc)$ by $g((a_n))=(a_n/(n+1))$. Then $g$ is completely contractive by the Max operator space structure on $\ell^1_{\bbn}(\bbc)$. Next, let $x=(1/n!)$. Then $x\in\ell^1_{\bbn}(\bbc)$ as $\norm{x}_{\ell^{1}}=e-1$. Now, let $x^*:Max(\ell^1_{\bbn}(\bbc))\to\bbc$ be a linear functional with $\norm{x^*}=1$ and $x^*(x)=\norm{x}_{\ell^{1}}$. Then there is no morphism $u:\ell^1_{\bbn}(\bbc)\to\bbc$ such that the following diagram commutes:
	
	\begin{center}
		\begin{tikzcd}
			Max(\ell^1_{\bbn}(\bbc)) \arrow[r,"x^*"] \arrow[d,"g", labels=left] & \bbc\\
			\ell^1_{\bbn}(\bbc) \arrow[from=2-1,to=1-2, "\enskip u",labels=below,dashed]
		\end{tikzcd}
	\end{center}

As from \cite[Proposition 2]{Pothoven69}, no such contraction exists, much less a complete contraction. Thus, $\bbc$ is not injective.
\end{proof}

\begin{thm}\label{piosp}
		In any category with operator spaces as objects and completely contractive maps as morphisms, the only injective object is $\{0\}$ (in the purely categorical sense).
\end{thm}

\begin{proof}
	If, for objects $A$ and $B$ in a category $C$, $A$ is a retract of $B$ and $B$ is injective, then $A$ is injective. Therefore, the result follows from the two previous lemmas.
\end{proof}

\begin{remark}
    This implies that in the case of the category consisting of $G$-operator spaces as objects and with $G$-equivariant completely contractive maps as morphisms, the only injective object is $\{0\}$.
\end{remark}

\begin{lemma}
	$\bbc$ is a retract of every operator system. 
\end{lemma}

\begin{proof}
	Let $X$ be an operator system, $\iota:\bbc\1\to X$ the inclusion map, and consider the map $\phi:\bbc\1\to\bbc$ defined by $\phi(\lambda\1)=\lambda$. Then by the Hahn-Banach theorem, there exists an extension $\tilde{\phi}:X\to\bbc$ such that $\norm{\tilde{\phi}}=\norm{\phi}$ and $\tilde{\phi}\circ\iota=\phi$. Note that $\tilde{\phi}$ is completely contractive. Also, since $\tilde{\phi}$ is a unital c.c. map between operator systems, $\tilde{\phi}$ is u.c.p. Identifying $\bbc\1$ with $\bbc$, it follows that $\bbc$ is a retract of $X$.
\end{proof}

\begin{lemma}
	In the category of operator systems and u.c.p. maps, $\bbc$ is not injective (in the purely categorical sense).
\end{lemma} 

\begin{proof}
	Let $c$ denote the $C^*$-algebra (considered as an operator system) of all convergent sequences, consider the projection onto the first coordinate  $\pi_1:c\to\bbc$, and consider $\phi:c\to c$ such that $\phi((a_n)_n)=(a_n/2)_n$ for sequences $(a_n)_n$ in $c_0$ and $\phi(\lambda(1,1,\dots))=\lambda(1,1,\dots)$. The map $\phi$ is unital and it is one-to-one and positive. Indeed, for one-to-one, let $(a_n)_n$ and $(b_n)_n$ be two sequences in $c$ with limits $\lambda_1\text{ and }\lambda_2$, respectively. Then $(a_n)=(c_n)+\lambda_1\1$ and $(b_n)=(d_n)+\lambda_2 \1$ for $(c_n)_n\text{ and }(d_n)_n$ in $c_0$. Then, assuming $\phi(a_n)=\phi(b_n)$, it follows that $\lambda_1=\lambda_2$ and $c_n=d_n$ for all $n$. Hence, $(a_n)_n=(b_n)_n$.
	
	For positivity, let $(a_n)_n\in c_{+}$ with limit $\lambda$. Then there exists a sequence $(b_n)_n\in c_0$ such that $a_n=b_n+\lambda$ for all $n$ and $b_n\in\bbr$ for all $n$. Since $b_n+\lambda\geq 0$ for all $n$, $b_n/ 2 \geq -\lambda/ 2\geq -\lambda$ for all $n$, which gives that $b_n/ 2 + \lambda \geq 0$ for all $n$. Thus $\phi(a_n)=(\frac{b_n}{2}+\lambda)_n\in c_+$. 
 
 So, it follows that $\phi$ is completely positive. But, there is no morphism $u:c\to\bbc$ such that the following diagram commutes:
	
	\begin{center}
		\begin{tikzcd}
			c \arrow[r,"\pi_1"] \arrow[d,"\phi", labels=left] & \bbc\\
			c \arrow[from=2-1,to=1-2, "\enskip u",labels=below,dashed]
		\end{tikzcd}
	\end{center}

If there was, it would mean that $u\arrowvert_{c_0}=2\pi_1$ and $u(\lambda(1,1,\dots))=\lambda(1,1,\dots)$. Therefore, $u((1-1/n)_n)$ is not positive, so $u$ is not u.c.p. Thus, $\bbc$ is not injective.
\end{proof}

\begin{thm}
	In the category of operator systems and u.c.p. maps, there are no injective objects (in the purely categorical sense).
\end{thm}

\begin{proof}
	If, for objects $A$ and $B$ in a category $C$, $A$ is a retract of $B$ and $B$ is injective, then $A$ is injective. Therefore, the result follows from the two previous lemmas.
\end{proof}

\begin{remark}
    The above proof also shows that there are no injective objects (in the purely categorical sense) in both the category consisting of unital $C^*$-algebra as objects and with u.c.p. maps as morphisms and the category consisting of unital operator spaces as objects and with u.c.c. maps as morphisms. It also implies that in the $G$-cases of these categories, there are no injective objects. 
\end{remark}

%-----Appendix by Dr. B-----%
\appendix \section{The unitization of an operator space\\ By David P. Blecher}\label{appendix}
 
 \bigskip
 
 Henceforth let $X$ be a real or complex operator space (the other categories considered in the earlier part 
 of the paper will be similar).  We will frequently refer to \cite{blm04} for background facts.   This text assumes complex scalars,
but  the analogous results hold for  real spaces (see e.g.\ \cite{b23}).

One may ask, first, if there is a sensible functor $\mcf_u$, `adjoining an identity',  from the category $\mathsf{OSp}$ of (real or complex) operator spaces to the category $\mathsf{OSp}^{\mathbbm{1}}$ of 
unital (real or complex) operator spaces.    Secondly, one may ask if there is such a functor $\mcf_u$ such that $I(\mcf_u(X))$ has a nice relation to $I(X)$, or is 
at least recognizable and useful for the analysis of the underlying spaces.  

Adjoining a unit to an operator space is a delicate business.  (Operator systems are much more reasonable in this regard, indeed there has been 
considerable interest in this topic recently, for example \cite{cvs21} and references therein.)  It is not hard to show that there is a `biggest'  but not a `smallest' unitization of an operator space.  (To 
see that there is no `smallest unitization', consider the embedding $j_n(z) = -(z,(1-\frac{1}{n}) z)$ of $\bbc$ in $l^\infty_2$.  This  makes 
$l^\infty_2 = \bbc (1,1) + j_n(\bbc)$ a unitization.   However $\| (1,1) + j_n(1) \| = \frac{1}{n} \to 0$; so that there is no proper unitization of $\bbc$ whose  norm  is 
dominated  by all of these.) 

By a biggest  unitization we mean  a `universal'  unitization with the biggest norm (and matrix norms).
This turns out to be  $X^1 = X \oplus^1 \bbf$, with identity $(0,1)$.  Assuming that this is a unital operator space it is clear from the universal 
property of the $1$-direct sum operator space that 
any linear complete contraction $u : X \to B(H)$  has a unital completely  contractive extension $X^1 \to B(H)$.
To see that $X^1$ is a unital operator space one may appeal to p.\ 163 in \cite{pisier03} in the complex case.  The real case follows from the 
complex case via the embedding $X \oplus^1 \bbr \subset X_c \oplus^1 \bbc = (X \oplus^1 \bbr)_c$ (see \cite[Lemma 5.3]{b23}).
Certainly this unitization defines a functor $\mathsf{OSp} \to \mathsf{OSp}^1$.
 It is  a faithful functor  (and is easily seen not to be full nor essentially surjective).

Unfortunately this universal  or `biggest' unitization $X^1$ has serious issues from some perspectives. 
For example our goal here is to connect to injective envelopes, 
and the injective envelope $B = I(X^1)$ of $X^1$ seems  too artificial and `big'  to be useful in many situations.   Even in trivial cases such as $X = \bbc$ this 
 injective envelope $B$ is  a large and complicated space (it is the injective envelope of the continuous functions on the circle).

 Instead, we will use a different and more tractable unitization $\mcf_u : \mathsf{OSp} \to \mathsf{OSp}^{\mathbbm{1}}$, namely {\em the unique operator algebra unitization} of $X$, 
 which actually maps into the subcategory of unital operator algebras.   We recall from Meyer's theorem \cite{blm04} that  any operator algebra has a unique operator algebra unitization.
 Moreover any operator space $X$ is an operator algebra with the 
 zero product; and its unique unitization above is $\mcf_u(X) = {\mathcal U}_{\rm d}(X)$.
   To explain this and to define this construction  we first recall that for $X \subset B(H)$, the 
Paulsen system is the operator system
$$ 
\mcs(X)\, = \,  \left[ \begin{array}{ccl}
\bbc I & X \\
X^\star   & \bbc I \end{array} \right] \, = \, \biggl\{\left[
\begin{array}{ccl} \lambda I & x\\ y^* & \mu I \end{array} \right] \
:\ x,y \in X,\ \lambda,\,\mu\in\bbc\biggr\}
$$
in $M_2(B(H))$, where 
$I = I_H$. 
We also define operator algebras, the subalgebras 
$${\mathcal U}(X)  \, = \,  \left[ \begin{array}{ccl}
\bbc I & X \\
0   & \bbc I \end{array} \right] , \; \; \; \; {\mathcal U}_{\rm d}(X)  \, = \, \biggl\{\left[
\begin{array}{ccl} \lambda I & x\\ 0 & \lambda I \end{array} \right] \
:\ x \in X,\ \lambda \in\bbc\biggr\} .
$$
Paulsen's lemma shows that $\mcs(X)$, and hence $\mcu(X)$ and $\mcu_d(X)$, only depends on the operator space structure
of $X$, and not on its representation on $H$ (see e.g.\ 1.3.14 and 2.2.10 --2.2.12 in \cite{blm04}).    The canonical complete isometry
$i_{12} : X \to {\mathcal U}_{\rm d}(X)$ has range which is an operator algebra (a subalgebra)  with the trivial product.  Its unique unitization (i.e.\ $i_{12}(X) + \bbc \, I_2$) is 
${\mathcal U}_{\rm d}(X)$.

 By \cite[Lemma 1.3.15]{blm04} we obtain that $\mcf_u(X) =  {\mathcal U}_{\rm d}(X)$ is a functor from  $\mathsf{OSp}$  to $\mathsf{OSp}^{\mathbbm{1}}.$ 
 This functor  is faithful, but not full nor essentially surjective.   To see that it is not full
let $A = {\mathcal U}_{\rm d}(\bbf)$ and consider the unital map $u : A \to A$ taking $i_{12}(x)$ to $\frac{x}{2} I_2$.
This will be a complete contraction if and only if the functional on $A$ taking $a = \lambda I_2 + i_{12}(x)$ to $\lambda  + \frac{x}{2}$ is contractive.
However this is just $w^* a w$ where $w$ is the column vector with equal entries $\frac{1}{\sqrt{2}}$.
(We remark however that ${\mathcal U}_{\rm d}(X)$ may also be viewed as a functor from   $\mathsf{OSp}$ to the category $\mathsf{OA}^{\mathbbm{1}}$
 of unital operator algebras and completely contractive unital homomorphisms,
and it is faithful and full into the latter category.  Indeed it is not hard to show that it is the left adjoint to the forgetful functor from $\mathsf{OA}^{\mathbbm{1}}$  to $\mathsf{OSp}$.)

It is not true in general that $I(\mcf_u(X)) = \mcf_u(I(X))$, since the injective envelope of a unital operator space is a $C^*$-algebra
 and ${\mathcal U}_{\rm d}(X)$  is not.   
 However the injective envelope of  ${\mathcal U}_{\rm d}(X)$ is a useful object that we now determine.  We recall that the injective envelope of 
 a (real or complex)  operator space $X$ is usually studied in terms of the four corners of the extremely 
 important  injective $C^*$-algebra ${\mathcal E}(X) = I({\mathcal S}(X))$.
 Indeed    the $1$-$2$-corner of ${\mathcal E}(X)$ is precisely $I(X)$, and the other three  corners also play a profound role in the study of $I(X)$
 (see e.g.\ \cite[Chapter 4]{blm04}, for example Section 4.4, and 4.4.2 there).

 \begin{thm} \label{id1}    
    If $X$ is a (real or complex)  operator space then $I(\mcf_u(X)) = I({\mathcal U}_d(X)) = {\mathcal E}(X)$.  
    Thus the $1$-$2$ corner of the injective envelope of  $\mcf_u(X)$ in $\mathsf{OSp}^{\mathbbm{1}}_{(\mathbb{C},\mathbb{C})}$ 
    is an injective envelope of $X$ in $\mathsf{OSp}_{(\mathbb{F},\mathbb{F})}$.
\end{thm}

In order to prove this we  first 
 prove an analogous result for the ternary envelope of $X$  that we believe is new and   useful.

 \begin{lemma} \label{dand12} Let $X$ be a (real or complex)  operator space.  If $(Z,i)$ is a ternary envelope of $X$ 
then $C^*_e({\mathcal U}_{\rm d}(X)) \cong {\mathcal L}(Z)^1$ via a $*$-isomorphism 
mapping $i_{12}(x)$ to $i(x)$ for $x \in X$.    Here ${\mathcal L}(Z)$ is the linking $C^*$-algebra (see e.g.\ \cite[Section 
 8.3]{blm04}); with ${\mathcal L}(Z)^1 
= {\mathcal L}(Z)$ if the latter is already unital.  \end{lemma}
 
   \begin{proof}  
   We note that ${\mathcal U}_{\rm d}(X)$ has the universal property that if $u : X \to B(K,H)$ is a complete contraction (resp.\ complete isometry) 
   then there exists a unital complete contraction (resp.\ complete isometry)  $\tilde{u} : {\mathcal U}_{\rm d}(X) \to B(H \oplus K)$ with $\tilde{u} \circ i_{12} = i_{12} \circ u$, where 
   $i_{12}$ denotes the 1-2-`corner map'.     This follows from  the universal property  of the Paulsen system.   We will not use this but
   these universal properties both
     characterize ${\mathcal U}_{\rm d}(X)$ (exercise). 

It follows from this universal property   and the universal property  of the $C^*$-envelope
that $C^*_e({\mathcal U}_{\rm d}(X))$ has the universal property that if $u : X \to B(K,H)$ is a complete isometry, and if $B$ is the $C^*$-subalgebra of $B(H \oplus K)$ 
generated  by the identity operator and $i_{12}(u(X))$, then there exists a unital surjective $*$-homomorphism $\pi : B \to C^*_e({\mathcal U}_{\rm d}(X))$ such that 
$\pi(i_{12}(u(x))) =  j (i_{12}(x))$  for $x \in X$, where $j :   {\mathcal S}_{\rm d}(X) \to C^*_e({\mathcal S}_{\rm d}(X)) = C^*_e({\mathcal U}_{\rm d}(X))$ is the canonical map.  

We claim that the latter property forces the TRO $W$ generated by 
$i_{12}(X)$ in $C^*_e({\mathcal S}_{\rm d}(X))$ to be a ternary envelope of $X$, i.e.\ effectively be $Z$.   To see this, let $E$ be the TRO generated by $u(X)$ in $B(K,H)$,
and note that $\pi \circ i_{12} : E \to C^*_e({\mathcal U}_{\rm d}(X))$ is a ternary morphism.  Since  
$\pi(i_{12}(u(x))) =  j (i_{12}(x))$ for $x \in X$,  the  range of $\pi \circ i_{12}$ 
is contained in $W$.   So $(W,i_{12})$ has   the universal property  of the  ternary envelope of $X$.  Hence there is a ternary isomorphism $\theta : W \to Z$ such that $\theta \circ i_{12}$ is the canonical 
embedding $i : X \to Z$.    

Next notice that if we take $u$ above to be  $i$, then the $*$-homomorphism $\pi$ above takes ${\mathcal L}(Z)^1$ onto $C^*_e({\mathcal U}_{\rm d}(X))$, with 
$\pi(i_{12}(i(x))) =  j (i_{12}(x)) \in W$.   Thus the ternary morphism $\pi \circ i_{12} : Z \to W$ is surjective.   It is also one-to-one
since  $\theta (\pi(i_{12}(i(x))) = \theta (j (i_{12}(x))) = i(x)$.   Indeed  this ternary morphism is $\theta^{-1}$.   Any ternary 
isomorphism $\rho : Z \to W$ between TRO's induces a unique $*$-isomorphism  ${\mathcal L}(Z) \to {\mathcal L}(W)$ 
mapping $i_{12}(z)$ to $\rho(z)$ for $z \in Z$ (see \cite[Corollary 8.3.5]{blm04}).   By this uniqueness, this $*$-isomorphism is the restriction of $\pi$ to ${\mathcal L}(Z)$.
If the latter is non-unital  it follows that 
$\pi$ is a $*$-isomorphism.     If it is unital then so is the $C^*$-algebra $D$ generated by 
$i_{12}(X)$ in $C^*_e({\mathcal U}_{\rm d}(X)$.  In this case $\bbf 1_D +  i_{12}(X) = {\mathcal U}_{\rm d}(X))$ completely isometrically
by Meyer's theorem if necessary, and 
 we may take $D = C^*_e({\mathcal U}_{\rm d}(X))$.   
\end{proof}  

\begin{proof}[Completion of proof of Theorem \ref{id1}:] 
    
   Let  $Z = {\mathcal T}(X)$, a TRO.  By the lemma, 
$$C^*_e({\mathcal U}_{\rm d}(X)) = {\mathcal L}(Z)^1,$$ and so $$I({\mathcal U}_d(X)) = I(C^*_e({\mathcal U}_{\rm d}(X))) =I({\mathcal L}(Z)^1) =
I({\mathcal L}(Z)) =
I({\mathcal S}(X)).$$ 
The second last equality follows from e.g.\ \cite[Corollary 4.2.8 (2)]{blm04}.  The last equality is in \cite[Corollary 1.8 (iii)]{bp01} but we give a proof to check its veracity in the real case too.
Note that  ${\mathcal L}(Z)$ is a $C^*$-subalgebra of $I({\mathcal S}(X))$.   The identity projections of $I_{11}(X)$ and $I_{22}(X)$ (see e.g.\ 4.4.2 in 
\cite{blm04}) 
act as multipliers of ${\mathcal L}(Z)$ inside $I({\mathcal S}(X))$.    We have ${\mathcal S}(X) \subset 
{\mathcal S}(Z) \subset M({\mathcal L}(Z)) \subset I({\mathcal S}(X))$.   The well known fact that  the multiplier algebra of a 
$C^*$-algebra $B$ is a $*$-sub-algebra of $I(B)$ was first noted by Hamana.  In the real case it can be seen from
$$M(B) \subset \mcm_\ell(B) \subset I_{11}(B) = I(B),$$
the last equality from 4.4.13 in \cite{blm04}. 
Moreover $B \subset B^1 \subset M(B) \subset I(B)$, and any UCP map on $I(B)$   fixing $M(B)$  also fixes $B^1$.  So by rigidity, 
$I(M(B)) = I(B)$.  Thus 
 $I(M({\mathcal L}(Z))) = I({\mathcal L}(Z))$.   Similarly, a UCP map on $I({\mathcal S}(X))$ fixing $M({\mathcal L}(Z))$  also fixes ${\mathcal S}(X)$.  So by rigidity, 
$I(M({\mathcal L}(Z))) = I({\mathcal S}(X))$.   Thus $I({\mathcal L}(Z)) =  I({\mathcal S}(X))$.
\end{proof}

The variant of the above results with the subscript $d$ deleted is well known, indeed we can use ${\mathcal U}(X)$ and  ${\mathcal S}(X)$ somewhat interchangeably, by Arveson's result \cite[Lemma 1.3.6]{blm04}. 
It is somewhat surprising that the injective and $C^*$-envelopes of ${\mathcal U}_d(X)$ (and  hence of ${\mathcal S}_d(X)$) is so large. 
We remark that $C^*_e({\mathcal U}_d(X)) \neq C^*_e(\mcs(X))$ in general (a counterexample is $X$ the $C^*$-algebra of compact operators).

We sketch proofs of how the above works out in some of the other categories studied in this paper.
For a complex operator space let $B$ be an injective envelope of  $\mcf_u(X)$ in $\mathsf{OSp}^{\mathbbm{1}}_{(\mathbb{C},\mathbb{R})}$.  
Then  $B$ is an injective envelope of  $\mcf_u(X)$ in $\mathsf{OSp}^{\mathbbm{1}}_{(\mathbb{R},\mathbb{R})}$ (it is real injective, and real rigid, see Corollary~\ref{f00injenv}).     
Its $1$-$2$ corner, which is
a complex injective operator space, is a 
 real injective envelope of $X$, and is easily seen to be an  injective envelope of $X$ in $\mathsf{OSp}_{(\mathbb{C},\mathbb{R})}$.
 Indeed it is complex injective by \cite{bck23}, and is easily seen to be $(\mathbb{C},\mathbb{R})$-essential. 

 Next let $X$ be in $\mathsf{OSp}^{G}_{(\mathbb{F},\mathbb{F})}$.  Then $\mcs(X)$ is a $G$-operator system, where the $G$-action by unital complete order 
isomorphisms is the extension of the 
 $G$-action on $X$ by \cite[Lemma 1.3.15]{blm04}.    Therefore ${\mathcal U}_{\rm d}(X)$ and ${\mathcal U}(X)$ are unital  $G$-operator spaces. 
  Any $G$-injective envelope  of ${\mathcal U}(X)$, which with its unique $C^*$-algebra structure 
  is by \cite[Proposition 4.9]{bck23} a $G$-operator system injective envelope of $\mcs(X)$, 
   is a unital $G$-$C^*$-algebra whose $1$-$2$ corner 
 is a $G$-injective envelope of $X$ (see \cite[Theorem 4.3]{hamana11} and its proof).   
 This equals $I_G({\mathcal U}_{\rm d}(X))$ if $G$ is finite.
 One ought to be able to show that the latter is also true for infinite $G$ 
 by modifying the proofs of A.1 and A.2 above.  In the case of A.1 this seems unproblematic, however   
 modifying Theorem A.2  is related to finding equivariant versions of the results cited in the proof of that theorem, 
 which seems like an  interesting project.   As above,  the  $1$-$2$ corner of the injective envelope $B$ 
 of  ${\mathcal U}(X)$ in $\mathsf{OSp}^{\mathbbm{1},G}_{(\mathbb{C},\mathbb{R})}$ is 
 an  injective envelope of $X$ in $\mathsf{OSp}^{\mathbbm{1},G}_{(\mathbb{C},\mathbb{R})}$. 

%-----Bibliography-----%
\bibliographystyle{abbrv}
\bibliography{main}

\end{document}

%% file: preamble.tex
\usepackage{amsmath}
\usepackage{amsfonts, amssymb, amsthm, bbm, upgreek}
\usepackage{thmtools, mathtools}
\usepackage{mathrsfs,physics,commath}
\usepackage{stackrel, array}
\usepackage[shortlabels]{enumitem}
\usepackage{xspace, xcolor, xpatch}
\usepackage{url, hyperref}
\usepackage{tikz-cd}
\newcommand{\mcb}{{\mathcal{B}}}
\newcommand{\mcc}{{\mathcal{C}}}

\newcommand{\mcf}{{\mathcal{F}}}

\newcommand{\mch}{{\mathcal{H}}}

\newcommand{\mcm}{{\mathcal{M}}}

\newcommand{\mcs}{{\mathcal{S}}}
\newcommand{\mct}{{\mathcal{T}}}
\newcommand{\mcu}{{\mathcal{U}}}

\newcommand{\bbc}{{\mathbb{C}}}

\newcommand{\bbf}{{\mathbb{F}}}

\newcommand{\bbn}{{\mathbb{N}}}

\newcommand{\bbr}{{\mathbb{R}}}

\newcommand{\acts}{\curvearrowright}
\newcommand{\1}{\mathbbm{1}}

\newcommand{\vphi}{\varphi}
\newcommand{\ii}{\textup{i}}

\let\norm\undefined % <-- "Undefine" \norm
\DeclarePairedDelimiter\norm{\lVert}{\rVert}

\newtheorem{theorem}{Theorem}

\theoremstyle{plain}
\newtheorem{thm}{Theorem}[section]
\newtheorem{cor}[thm]{Corollary}
\newtheorem{lemma}[thm]{Lemma}
\newtheorem{prop}[thm]{Proposition}
\theoremstyle{definition}
\newtheorem{defn}[thm]{Definition}
\newtheorem{remark}[thm]{Remark}

\hypersetup{
	colorlinks=true,
	linktoc=all,     
	linkcolor=blue,
	citecolor=blue
}